\documentclass{newarticle}
\usepackage[latin1]{inputenc}
\usepackage{amssymb}
\usepackage{graphicx}
\usepackage{amsmath}
\usepackage{amsthm}
\usepackage{mathrsfs}
\usepackage{accents}
\usepackage{subfigure}

\newtheorem{theorem}{Theorem}
\newtheorem{lemma}[theorem]{Lemma}
\newtheorem{proposition}{Proposition}
\newtheorem{remark}{Remark}

\theoremstyle{definition}
\newtheorem{definition}{Definition}
\newtheorem{assumption}[definition]{Assumption}

\newcommand{\R}{\mathbb R}
\newcommand{\bG}{\mathbf{G}}
\newcommand{\f}[1]{\mathbf{#1}}
\newcommand{\PiOmega}{\Pi_{\S_h}}
\newcommand{\PiSigma}{\Pi_{\V_h}}
\renewcommand{\S}{ \mathcal{S}}
\newcommand{\V}{ \mathcal{V}}
\newcommand{\refext}[1]{\widetilde{#1}}

\newcommand{\refbox}{R}
\newcommand{\spline}{\sigma}
\newcommand{\rst}[1]{\ensuremath{ \left.\rule{0ex}{1.5ex}\right|_{#1}}}
\newcommand{\bzeta}{{\boldsymbol{\zeta}}}
\newcommand{\T}{\mathcal{T}}
\newcommand{\id}{\ell}
\newcommand{\bxi}{\boldsymbol{\xi}}
\newcommand{\bA}{{\bf A}}
 \newcommand{\A}{\mathcal{A}}
 \newcommand{\B}{\mathcal{B}}

\newcommand\twoD[1]{\accentset{\diamond}{#1}}
\newcommand\oneD[1]{\bar{#1}}
\newcommand\inter[1]{\emph{interior}\left({#1}\right)}

\begin{document}

\begin{frontmatter}
  
\title{Unstructured spline spaces for isogeometric analysis based on spline manifolds}

\author[1,2]{Giancarlo Sangalli}
\author[1]{Thomas Takacs}
\author[2]{Rafael V\'azquez}

\address[1]{Dipartimento di Matematica ``F. Casorati'', Universit\`a degli Studi di Pavia, Italy}
\address[2]{Istituto di Matematica Applicata e Tecnologie Informatiche 
``E. Magenes'' (CNR), Italy}

\begin{abstract}
Based on \cite{grimm1995modeling} we introduce and study  a
mathematical framework for analysis-suitable unstructured  B-spline
spaces. In this setting the parameter domain has a manifold structure 
which allows for the definition of function spaces that have a 
tensor-product structure locally, but not globally. This includes 
configurations such as B-splines over multi-patch domains with
extraordinary points, analysis-suitable unstructured T-splines,
or more general constructions. Within this framework, we
generalize  the concept of dual-compatible
B-splines (developed for structured T-splines in
\cite{beirao2013analysis}). This allows us to prove the key properties
that are needed for isogeometric analysis, such as linear independence and
optimal approximation properties for $h$-refined meshes. 
\end{abstract}

\end{frontmatter}

\section{Introduction}

In isogeometric analysis, as it was introduced in \cite{Hughes2005}, the physical domain is usually parametrized smoothly by tensor-product B-splines or NURBS, therefore it has to be diffeomorphic to a rectangle. When dealing with geometrically complex domains, there is a need for a more flexible geometry representation. Several approaches have been proposed so far. 
We would like to point out several common strategies that can be studied within the framework we present in this paper. One is to use a multi-patch representation of the domain, i.e. to split the domain into rectangular or hexahedral boxes. Starting from the framework presented in \cite{igaBook}, there exist several contributions on multi-patch representations in isogeometric analysis \cite{kiendl2010,Beirao2011,Kleiss2012,Xu2013,Juettler2014,Buchegger2015}. In most constructions, the basis functions are only $C^0$ across patch interfaces. We want to point out the recent paper \cite{Buchegger2015}, in which the authors present a construction to enhance the smoothness across patch interfaces away from extraordinary features. 

Another way to handle complex geometries is to use a specific set of
non-tensor product function spaces, such as unstructured T-splines
(see \cite{wang2011converting,wang2012converting} as well as
\cite{Scott2013}). In both cases, the resulting function spaces are
very similar to the ones proposed in Sections
\ref{sec:dual-basis-construction} and \ref{sec:properties-dual-compatible}. We however follow a slightly different approach to represent the parameter domain and define the function spaces over it. 

Smooth constructions over general quad meshes can be introduced by means of subdivision surfaces. They have been studied in the context of isogeometric analysis in \cite{burkhart2010,barendrecht2013,juttler2015}. We do not go into the details of subdivision based constructions, but want to point out the relation to the manifold based framework. Function spaces based on subdivision schemes that generalize tensor-product B-splines, such as Doo-Sabin (see \cite{Doo1978}) for biquadratics or Catmull-Clark (see \cite{Catmull1978}) for bicubics, can be interpreted as spline manifold spaces. In both cases, the function space is spanned by standard tensor-product B-splines as well as certain special functions near extraordinary points. The special functions are piecewise polynomials over infinitely many rings of quadrilateral elements.

In this paper we consider spline manifold spaces as a theoretical tool
in isogeometric analysis. Note that it is not the purpose of this
study to introduce a new way to define or implement unstructured
spline spaces. The aim is rather to develop a framework to study the
theoretical properties of a wide range of spline representations over
non-rectangular parameter domains based on manifolds. Concerning
spline manifolds we refer to \cite{grimm1995modeling}, where the
authors propose  an abstraction of the concept of geometry
parametrizations that is based on a manifold structure. In the proposed framework, 
the underlying parameter domain forms a manifold, thus allowing to 
represent more general geometric domains.

In the following we briefly review the concept of a geometric domain which is parametrized by spline manifolds, and discuss its main
features. We consider a domain $\Sigma \subset \R^n$ which can be 
interpreted as a $d$-dimensional manifold with $d\leq n$. We mostly focus on 
a two-dimensional planar domain $\Sigma \subset
\R^2$, a surface $\Sigma \subset \R^3$ or a three-dimensional volumetric domain $\Sigma \subset
\R^3$. In an abstract setting, a manifold $\Sigma$ is defined by an \emph{atlas}, i.e., a family of charts $\Sigma_i$ such that
\begin{displaymath}
  \Sigma = \bigcup_{i=1}^{N} \Sigma_{i},
\end{displaymath}
together with suitable transition maps between intersecting charts $\Sigma_i$ and $\Sigma_j$.
To define $\Sigma$ as a spline manifold, we must define a family of \emph{open} parameter subdomains $\Omega_i$, together with a spline space on each subdomain, and each chart is the image of a \emph{spline parametrization} $\bG_i: \Omega_i \rightarrow \Sigma_i$. The union of the subdomains $\Omega_i$ forms an unstructured parameter domain $\Omega$, and suitable transition functions between the subdomains, together with a relation between the corresponding spline spaces, endow $\Omega$ with a manifold structure, which is inherited by $\Sigma$. Notice that to cover completely the domain $\Omega$, the open parameter subdomains $\Omega_i$ overlap. Hence, in order to cover the extraordinary points (or edges) it is necessary that some of the subdomains $\Omega_i$ also have an unstructured configuration, and non-tensor product spline functions need to be defined. This is in contrast to traditional multi-patch representations, where only the patch boundaries intersect, and the (closed) domain is defined as the union of the closure of the patches.

In Section \ref{sec:parameter-manifold} we recall and study the notion of a parameter manifold as presented in \cite{grimm1995modeling}. We introduce the mesh on the parameter manifold and corresponding mesh constraints in Section \ref{sec:mesh}. In Section \ref{sec:splinemanifoldspaces} we introduce spline manifold spaces based on tensor product B-splines and relate them to existing constructions \cite{Buchegger2015,wang2011converting,wang2012converting,Scott2013}. Finally, we develop the framework of analysis suitable spline manifold spaces in Section \ref{sec:analysis-suitable}, where we extend the notion of dual-compatibility to B-spline manifolds. In the main part of the paper we assume that the manifold has no boundary. For a more detailed study of spline manifolds with boundary see Appendix \ref{appendix:boundary}. We conclude the paper and present possible extensions in Section \ref{sec:conclusion}.

\section{Parameter manifold}
\label{sec:parameter-manifold}

Before we can define unstructured spline spaces on manifolds we need to introduce the 
abstract representation of the parameter domain, the so-called \emph{parameter manifold}. 
The following definitions are taken from \cite{grimm1995modeling}. The definitions are 
valid for arbitrary dimension $d$, but we will mostly consider $d=1,2,3$.

\begin{definition}[Proto-manifold]\label{defi:proto-manifold} A proto-manifold of dimension $d$
consists of
  \begin{itemize}
  \item a finite set $\{  \omega_i\}_{i=1,\ldots,N}$  (named
    \emph{proto-atlas}) of \emph{charts} $\omega_i$, that are open polytopes $ \omega_i \subset \R^d$, 
    line segments for $d=1$, polygons for $d=2$ or polyhedra for $d=3$;
  \item a set of open \emph{transition domains}  $\{  \omega_{i,j}\}_{i,j=1,\ldots,N}$
    such that $  \omega_{i,j} \subset   \omega_{i} $ and $
    \omega_{i,i} =  \omega_{i} $; 
  \item a set of \emph{transition functions} $\{  \psi_{i,j}\}_{i,j=1,\ldots,N}$, that are
    homeomorphisms  $\psi_{i,j} : \omega_{i,j}  \rightarrow  \omega_{j,i} $ fulfilling 
     the \emph{cocycle condition} $\psi_{j,k}\circ \psi_{i,j} =\psi_{i,k} $ in $ \omega_{i,j}  \cap  \omega_{i,k}  $ for all $i,j,k=1,\ldots,N$; 
  \item For every $i,j$, with $i\neq j$, for every $\boldsymbol{\zeta}_i \in \partial \omega_{i,j} \cap \omega_{i}$ and $\boldsymbol{\zeta}_j \in \partial \omega_{j,i} \cap \omega_{j}$, there are open balls, $V_{\boldsymbol{\zeta}_i}$ and $V_{\boldsymbol{\zeta}_j}$, centered at $\boldsymbol{\zeta}_i$ and $\boldsymbol{\zeta}_j$, such that no point of $V_{\boldsymbol{\zeta}_j} \cap \omega_{j,i}$ is the image of any point of $V_{\boldsymbol{\zeta}_i}\cap\omega_{i,j}$ by $\psi_{j,i}$.

  \end{itemize}
  
 \end{definition}
\begin{figure}[!ht]
    \centering
    \includegraphics[trim= 20pt 70pt 20pt 50pt,clip=true,width=.8\textwidth]{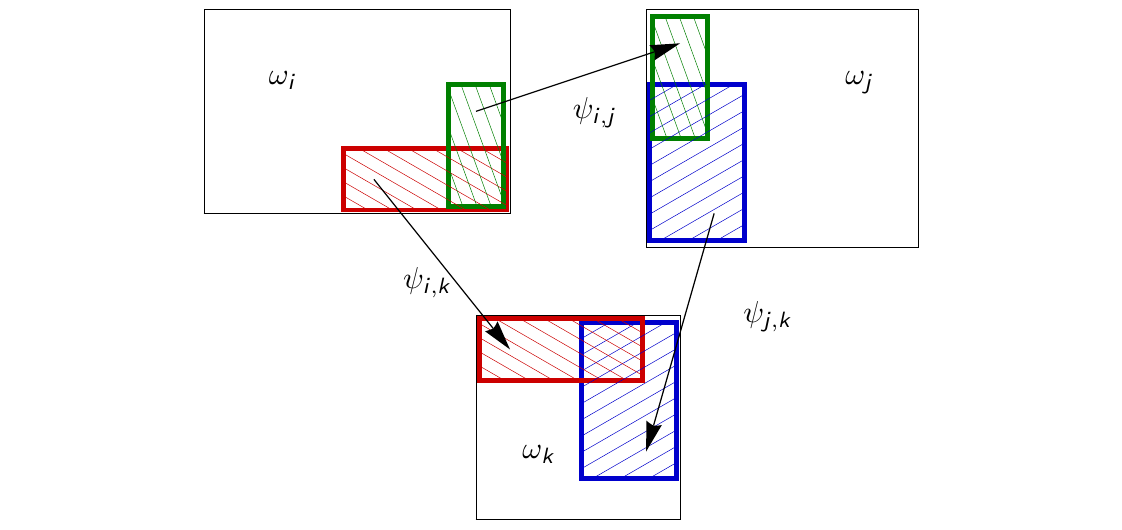}
    \caption{Visualization of the cocycle condition}\label{fig:cocyclecond}
\end{figure}
In Figure \ref{fig:cocyclecond} we visualize the cocycle condition, depicting three domains $\omega_i$, $\omega_j$ and $\omega_k$ and 
respective transition functions $\psi_{i,j}$, $\psi_{j,k}$ and $\psi_{i,k}$. The hatched regions represent corresponding transition domains. 

In general, the transition domains $\omega_{i,j}$ may be empty. It follows  directly from the
cocycle condition  that the transition function $\psi_{i,i} $ is the identity function on $\omega_{i,i} = \omega_{i}$ and that 
$\psi_{i,j}^{-1}  = \psi_{j,i} $ for all $i,j$. 

\begin{remark}
The last condition in Definition \ref{defi:proto-manifold} is taken from \cite{Siqueira2009}. It guarantees that the 
proto-manifold actually represents a manifold, i.e. that there are no bifurcations of the domain. Note that 
this condition is not necessary for the following definitions. However, if the condition is omitted, the resulting object is not a 
manifold anymore.
\end{remark}

In the next section the charts $ \omega_i$ will be used as (local) parameter domains to define
splines. By merging and identifying the charts of the  proto-manifold we obtain
a manifold, which will serve as the parameter domain in our setting, thus the name \emph{parameter manifold}.
  \begin{definition}[Parameter manifold] \label{def:parameter-manifold}
Given a proto-manifold, the set 
    \begin{equation}\label{eq:disj-union-of-Omega-i}
       \Omega = \left ( \bigsqcup_{i=1,\ldots,N}  \omega_i\right )
      \bigg/\sim
    \end{equation}
is called a \emph{parameter manifold}. 
    Here $\bigsqcup$ denotes the disjoint union, i.e., 
$$\bigsqcup_{i=1,\ldots,N}  \omega_i = \left \{[\bzeta_i,i] , \,
  \bzeta_i\in \omega_i,i=1,\ldots,N \right \}
$$
 and the equivalence relation $\sim$ is defined for all $\bzeta_i\in
     \omega_i$ and $\bzeta_j\in  \omega_j$, as
$$
[\bzeta_i,i] \sim
     [\bzeta_j,j] \Leftrightarrow \psi_{i,j}(\bzeta_i)=\bzeta_j.
$$
We denote by $\pi_{i}(\bzeta_i) \in \Omega$ the equivalence class corresponding to $\bzeta_i \in  \omega_i$. 
  \end{definition}

To motivate the definition (and wording) above, we recall that $
\Omega $ is indeed a topological manifold, endowed with the
natural  structure (see \cite{grimm1995modeling} and \cite{Siqueira2009}): $\pi_{i}$ is a one-to-one correspondence
between each $ \omega_i$ and  $\pi_{i}( \omega_i)$ and plays the
role of a local representation, its inverse $\pi_i^{-1}$ is the coordinate chart
in the classical language of manifolds. 
In the following we use the notation  $ \Omega_i = \pi_{i}( \omega_i)$ and  $ \Omega_{i,j} = \pi_{i}( \omega_{i,j})$, where  $  \Omega_i,\Omega_{i,j} \subset \Omega$. Hence we have $ \Omega_{i,j} = \Omega_{j,i} =  \Omega_{i} \cap \Omega_{j}$. The parameter manifold $\Omega$ can be interpreted as a topological manifold. 

\begin{remark}
To be precise, a parameter manifold is a class of piecewise
smooth manifolds, which is a sub-class of topological
manifolds and a super-class of smooth manifolds. It is similar
to the concept of piecewise linear manifolds (see \cite{Rourke1972}). Depending on its 
local structure, the parameter manifold is either $C^0$ or $C^\infty$ locally. In general the parameter manifold is globally  $C^0$ and
piecewise $C^\infty$.
\end{remark}

We assume that there exists a metric on $\Omega$, that allows us to
introduce the usual Lebesgue space  $L^2(\Omega)$. 
In Section \ref{sec:analysis-suitable} we present in more detail the function spaces that are necessary for the analysis.

Having defined a parameter manifold $\Omega$, we can define a mesh $\T$ on $\Omega$ in Section~\ref{sec:mesh} and a function space with a piecewise construction over the mesh in Section~\ref{sec:splinemanifoldspaces}.

\section{Mesh and mesh constraints on the parameter manifold}\label{sec:mesh}

We introduce the general concept of a mesh $\T$ on a parameter manifold $\Omega$ in Section \ref{subsec:mesh-para-manif} and then develop the specific configurations we consider in Section \ref{sec:charts}. Following that, we present some example configurations and study the meshing of complex geometries in Section \ref{sec:meshing}.

\subsection{Mesh on a parameter manifold}
\label{subsec:mesh-para-manif}

First we define a proto-mesh on the charts. 
\begin{definition} \label{def:proto-mesh}
A \emph{proto-mesh} 
 \begin{equation}
   \label{eq:mesh-i}
   \{\tau_i \}_{i=1,\ldots,N} \hspace{10pt} \mbox{ with } \tau_i= \{ q \subset  \omega_i \}
 \end{equation}
is a collection of sets, where 
\begin{itemize}
\item each set $\tau_i$ is composed of open
polytopes $q$, called \emph{elements}; these are intervals,
quadrilaterals, hexahedra, etc., depending on the dimension $d=1,2,3,\ldots$, respectively, and each set $\tau_i$ 
is a mesh on $\omega_i$, i.e. the elements are disjoint 
and the union of the closures of the elements is the closure of $
\omega_i $;
\item for every $i,j$, $\omega_{i,j}$ is the interior of the 
  union of the closure of elements of $\tau_i$; and 
\item the transition functions $\psi_{i,j}$ map elements onto elements, i.e. 
\begin{equation}
  \label{eq:mesh-compatile-with-transitions}
  \forall q \in \tau_i, \,  \psi_{i,j} (q) \in \tau_j.
\end{equation}
\end{itemize}
\end{definition}

Furthermore, we have the following.
\begin{assumption}\label{assumption:piecewise-multilinear-manifold}
The transition functions are piecewise
$d$-linear mappings with respect to the mesh.
\end{assumption}
The proto-mesh naturally defines a mesh on $\Omega$.
\begin{definition}[Mesh on $\Omega$] We define the mesh on the
  parameter manifold  $\Omega$ as
\begin{equation}
  \label{eq:mesh}
  \T= \{ Q \subset \Omega : \; Q = \pi_{i}(q), q \in    \tau_i, i=1,\ldots,N \}.
\end{equation}
\end{definition}
\begin{remark}\label{rem:mesh-on-Omega}
  Thanks to \eqref{eq:mesh-compatile-with-transitions}, the set $\T$ in 
  \eqref{eq:mesh} is indeed a well defined mesh on $ \Omega $. 
The elements of $\T$ are subsets of $\Omega$ that fulfill the 
standard properties of a mesh, i.e. the elements are disjoint, the union of the closures of the 
elements is the closure of $\Omega$. 
Since $ \Omega $
is a topological manifold the notion of the closure of elements, boundary edges, faces, etc. is 
well defined and derives from the local definition on each chart and the equivalence relation given 
by the transition functions.
\end{remark}

\subsection{Structured and unstructured meshes on the parameter manifold}\label{sec:charts}

In the following we  classify and restrict
to specific but relevant charts, depending on their local  mesh topology. 

We will say that several objects \emph{share} a common object A, if A is contained in the closure of each of
the objects.  Mesh objects and properties  defined on the charts can be carried over to the 
mesh $\T$ on $\Omega$. We will not distinguish between the mesh $\T$ on $\Omega$ and the proto-mesh defined on the charts, unless it is necessary. The relations between vertices, edges, faces and elements are stated within the global mesh. Moreover, these geometric objects are always assumed to be open.

\begin{definition}[Structured chart]\label{def:structured-chart}
A $d$-dimensional chart $\omega_i$ is called a \emph{structured chart} if
\begin{itemize}
\item[(A)] 	 $\omega_i$ is a $d$-box, that is 
\begin{equation*}
	\omega_i=  \prod_{\id=1}^d \left] a_{i,\id},b_{i,\id} \right[ ,
\end{equation*}
and $\tau_i$ is a box mesh (see, e.g., \cite{Dokken2013});
\item[(B)] for every $j$,  $\omega_{i,j}$ is a $d$-box or the union of $d$-boxes with disjoint closures; and 
\item[(C)] for every $j$,  if $\omega_j$ is a  structured chart then the 
transition function $\psi_{i,j} : \omega_{i,j}  \rightarrow
  \omega_{j,i} $ is an affine mapping (a linear polynomial) in each connected component of $\omega_{i,j}$.
\end{itemize}
\end{definition}
Considering unstructured charts we need to distinguish three different types, depending on the dimension and on the topological structure.

\begin{definition}[Unstructured vertex chart, two-dimensional]\label{def:unstructured-chart-2d}
A two-dimensional chart $\omega_i$ is called an \emph{unstructured vertex
  chart} corresponding to an extraordinary vertex
$\bxi_i$ of \emph{valence} $k_i\neq 4$
 if 
\begin{itemize}
\item[(A)] 	 $ \omega_i $ is formed by a ring of $k_i$ conforming quadrangular segments $s_{i,\ell}$, with $\ell=1,\ldots,k_i$, around $\bxi_i$, where each segment  
  has a mesh $\sigma_{i,\ell}$ which is topologically equivalent to a box mesh, and the mesh $\tau_i$ is given as the union of the meshes $\sigma_{i,\ell}$ on the segments $s_{i,\ell}$; 
\item[(B)] $\omega_i$, except for the extraordinary vertex $\bxi_i$, is 
covered by the transition domains $\omega_{i,j}$ with structured charts, i.e. 
  \begin{displaymath}
    \bigcup _{\stackrel{j=1,\ldots,N}{\omega_j\textrm{ structured}}} \omega_{i,j} = \omega_i \setminus \bxi_i,
  \end{displaymath}
and the transition domains $\omega_{i,j}$ are given by the union of segments; and 
\item[(C)] if $ \omega_j$ is an unstructured chart with
  $i\neq j$, then $\omega_{i,j}= \emptyset$.
\end{itemize}
\end{definition}

\begin{figure}[!ht]
\centering
\subfigure[]{\includegraphics[width=0.25\textwidth]{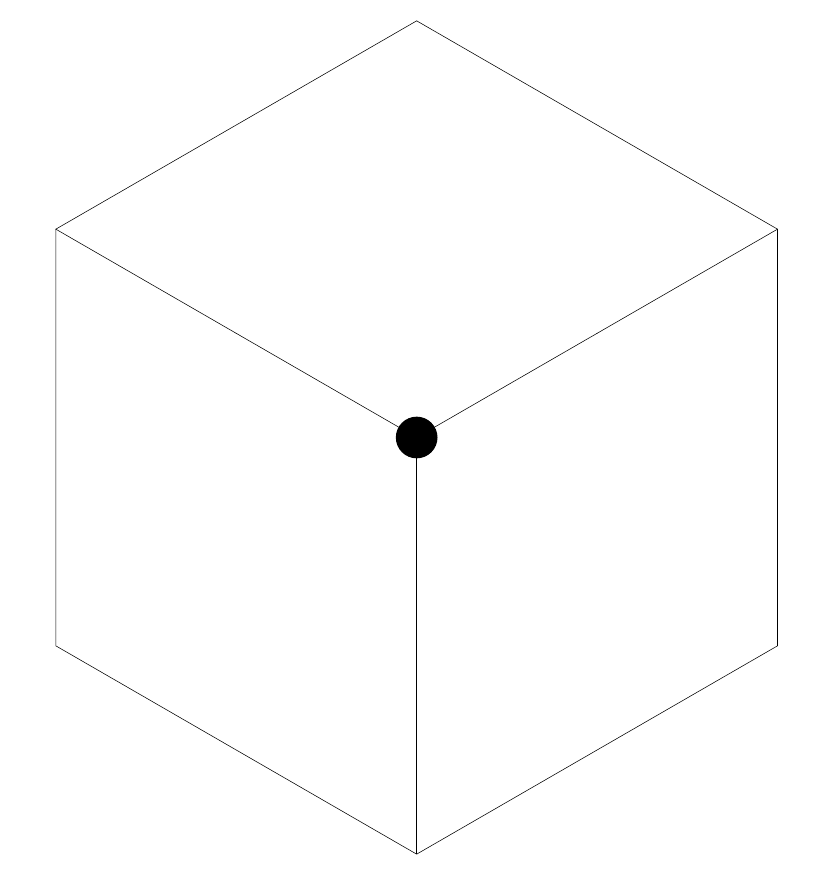}\label{fig:uv-charts-a}}
\hspace{0.05\textwidth}
\subfigure[]{\includegraphics[width=0.25\textwidth]{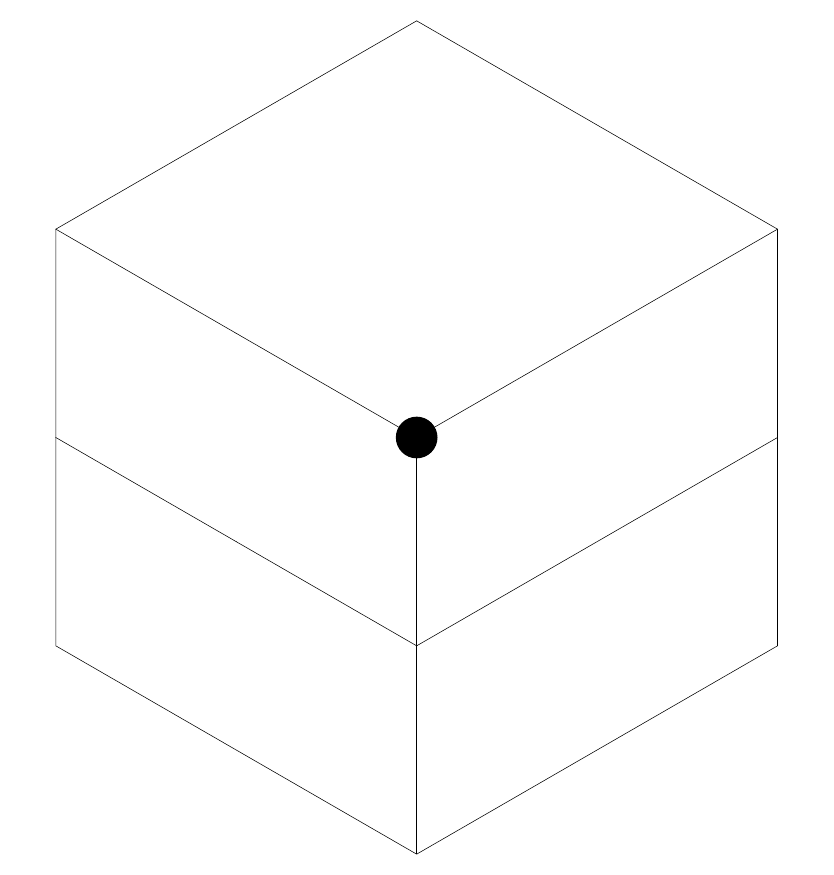}\label{fig:uv-charts-b}}
\hspace{0.05\textwidth}
\subfigure[]{\includegraphics[width=0.25\textwidth]{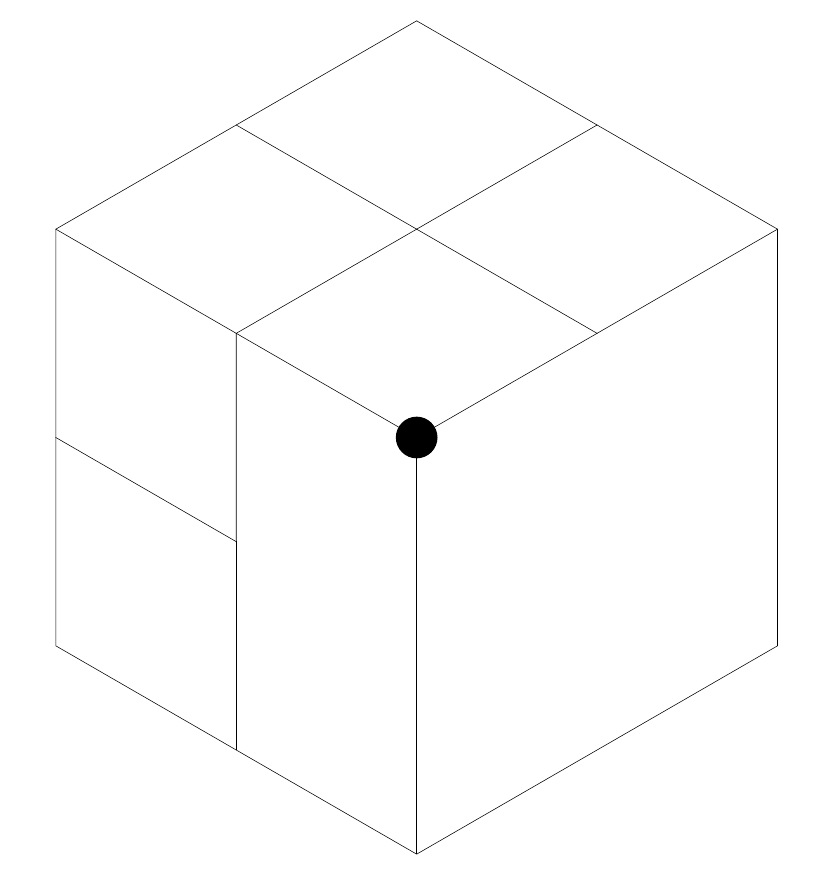}\label{fig:uv-charts-c}}
    \caption{Examples of meshes on a ring of $3$ conforming segments around a two-dimensional
      extraordinary vertex of valence $3$}\label{fig:uv-charts}
\end{figure}
To give more insight into the definition of an unstructured vertex
chart via its segments, we present some example configurations, which
are all constructed from the same three conforming segments. 
Figure~\ref{fig:uv-charts-a} depicts a mesh where each segment is covered by a single 
element. Figure~\ref{fig:uv-charts-b} depicts a conforming mesh and
Figure~\ref{fig:uv-charts-c} contains two hanging vertices, hence it
is a non-conforming mesh. This figure also tells that the
non-conformity of the mesh over the unstructured vertex chart may have
two reasons, either the meshes on two neighbouring segments do not
match (upper right hanging vertex) or the T-node is already present
within the mesh $\sigma_{i,\ell}$ on the segment (lower left hanging
vertex). 

\begin{definition}[Unstructured edge chart]\label{def:unstructured-edge-chart-3d}
A three-dimensional chart $\omega_i$ is called an \emph{unstructured
  edge chart} of valence $k_i \neq 4$ if 
\begin{itemize}
\item[(A)] 	 $\omega_i = \twoD{\omega}_i \times \oneD{\omega}_i $, where 
the two-dimensional chart $\twoD{\omega}_i$ is a two-dimensional unstructured vertex chart, 
as in Definition \ref{def:unstructured-chart-2d}, with an 
 extraordinary point $\twoD\bxi_i$ of valence $k_i$, partitioned into two-dimensional segments $\twoD s_{i,\ell}$ for $\ell = 1,\ldots,k_i$; 
$\oneD{\omega}_i$ is an interval $\oneD{\omega}_i  = \left]a_{i,3},b_{i,3}\right[$; each three-dimensional segment 
$s_{i,\ell} = \twoD s_{i,\ell} \times \oneD{\omega}_i$ has a mesh $\sigma_{i,\ell}$ which is equivalent to a three-dimensional box mesh, and the mesh $\tau_i$ is again given as the union of all meshes $\sigma_{i,\ell}$ on the segments;
\item[(B)] $\omega_i $, except for the \emph{extraordinary line}
  $\twoD\bxi_i \times\oneD{\omega}_i$, is 
covered by transition domains $\omega_{i,j}$ with structured charts, i.e. 
  \begin{displaymath}
    \bigcup _{\stackrel{j=1,\ldots,N}{\omega_j\textrm{ structured}}} \omega_{i,j} = \omega_i \setminus \twoD\bxi_i \times\oneD{\omega}_i,
  \end{displaymath}
and the transition domains $\omega_{i,j}$ are given as the tensor-product of the union of two-dimensional segments with an interval in the third direction; 
\item[(C)]  for structured charts $\omega_j$ the transition function $\psi_{i,j}$ is linear in the third coordinate $\zeta_3$; and 
\item[(D)] if $ \omega_j$ is an unstructured edge chart with
  $i\neq j$ and $\omega_{i,j} \neq \emptyset$, then there exists a structured chart $\omega_k$ such that $\omega_{i,j} \subseteq \omega_{i,k}$.
\end{itemize}
\end{definition}

\begin{definition}[Unstructured vertex chart, three-dimensional]\label{def:unstructured-vertex-chart-3d}
A three-dim\-ensional chart $\omega_i$ is called an \emph{unstructured vertex
  chart} corresponding to an extraordinary vertex
$\bxi_i$ of \emph{valence} $k_i$ 
 if 
\begin{itemize}
\item[(A)] $ \omega_i $ is formed by a ring of $k_i$ conforming hexahedral segments $s_{i,\ell}$, with $\ell=1,\ldots,k_i$, around $\bxi_i$, where each segment has a mesh $\sigma_{i,\ell}$ which is topologically equivalent to a box mesh, and the mesh $\tau_i$ is given as the union of the meshes $\sigma_{i,\ell}$ on the segments $s_{i,\ell}$;
\item[(B)] $\omega_i$, except for the
  extraordinary vertex $\bxi_i$, is 
covered by the transition domains $\omega_{i,j}$ with unstructured edge charts, i.e.,
  \begin{displaymath}
    \bigcup _{\stackrel{j=1,\ldots,N}{\omega_j\textrm{unstructured
          edge chart}}} \omega_{i,j} = \omega_i \setminus \bxi_i,
  \end{displaymath}
and the transition domains $\omega_{i,j}$ are given by the union of segments; and 
\item[(C)] if $ \omega_j$ is an unstructured vertex chart with
  $i\neq j$, then $\omega_{i,j}= \emptyset$.
\end{itemize}
\end{definition}

Note that from Definition \ref{def:unstructured-edge-chart-3d} (B) and Definition \ref{def:unstructured-vertex-chart-3d} (B) it follows that for every unstructured vertex chart in 3D the transition domains with structured charts cover everything, except for the extraordinary features (union of extraordinary vertex and extraordinary edges). In particular, the interior ef every element is covered.

We would like to point out that Definitions \ref{def:unstructured-chart-2d} (C), \ref{def:unstructured-edge-chart-3d} (D) and \ref{def:unstructured-vertex-chart-3d} (C) are not strictly necessary, but technicalities to simplify the definition of spline manifold spaces in Section \ref{sec:splinemanifoldspaces}. They guarantee that unstructured vertex charts do not overlap and that unstructured edge charts only overlap in the vicinity of an unstructured vertex. Notice that it is not possible that one element of the mesh $\T$ is adjacent to two unstructured vertices. In the context of meshing, this is not a severe restriction as we point out in the following section. Moreover, each unstructured vertex and each unstructured edge corresponds to exactly one unstructured vertex chart or unstructured edge chart, respectively.

\begin{assumption}
We assume that each chart  $\omega_i$   can be either 
structured as in Definition \ref{def:structured-chart} or unstructured as in 
Definitions \ref{def:unstructured-chart-2d}, \ref{def:unstructured-edge-chart-3d} and \ref{def:unstructured-vertex-chart-3d}.
\end{assumption}

In the following section we give some example configurations and study in more detail the meshing of complex geometries.

\subsection{Example configurations and meshing of complex geometries} \label{sec:meshing}

For simplicity, in the classification above we have restricted ourselves to a limited number of types of charts, which nevertheless contain most mesh configurations of practical interest. 
We first give a summary of the different types of
vertices that can occur (depending on the dimension), which motivates
our chart classification. Each unstructured chart is formally defined in such a way, that it 
can be used to cover a certain type of unstructured vertex. 
Following the discussion of the types of vertices, we present some simple example configurations and discuss the issues of meshing with manifolds. Without beeing thorough, we present configurations that can be covered and such that cannot. 
Since we do not want to go into the details of meshing we refer to the following literature
for quad-meshing \cite{Owen1998} and hex-meshing \cite{Owen1998,Juettler2014,Nguyen2014}.

For 1D domains (intervals, planar or spatial curves) there are only regular vertices, which are equivalent to the knots in the classical B-spline language. Hence no unstructured 
charts are needed. 

For 2D domains we consider three types of vertices. 
A 2D vertex can be a \emph{regular vertex},
a \emph{hanging vertex} (a T-node), or an \emph{extraordinary vertex} (see Table \ref{tab:classification-2d-vertices} and Figure \ref{fig:vert2D}). 
Both regular and hanging vertices are covered by structured charts, although they may also belong to an unstructured chart. 
Unstructured vertex charts cover a neighborhood of an extraordinary vertex. 

\begin{table}[ht]
\centering
\begin{tabular}{llr}
structured vertex & regular vertex & (Figure \ref{fig:vert2D-a}) \\
 & hanging vertex & (Figure \ref{fig:vert2D-b}) \\
unstructured vertex & extraordinary vertex &(Figures~\ref{fig:vert2D-c} and~\ref{fig:vert2D-d})
\end{tabular}
    \caption{Classification of vertices in 2D}\label{tab:classification-2d-vertices}
\end{table}

\begin{figure}[!ht]
\centering
\subfigure[]{\includegraphics[width=0.15\textwidth]{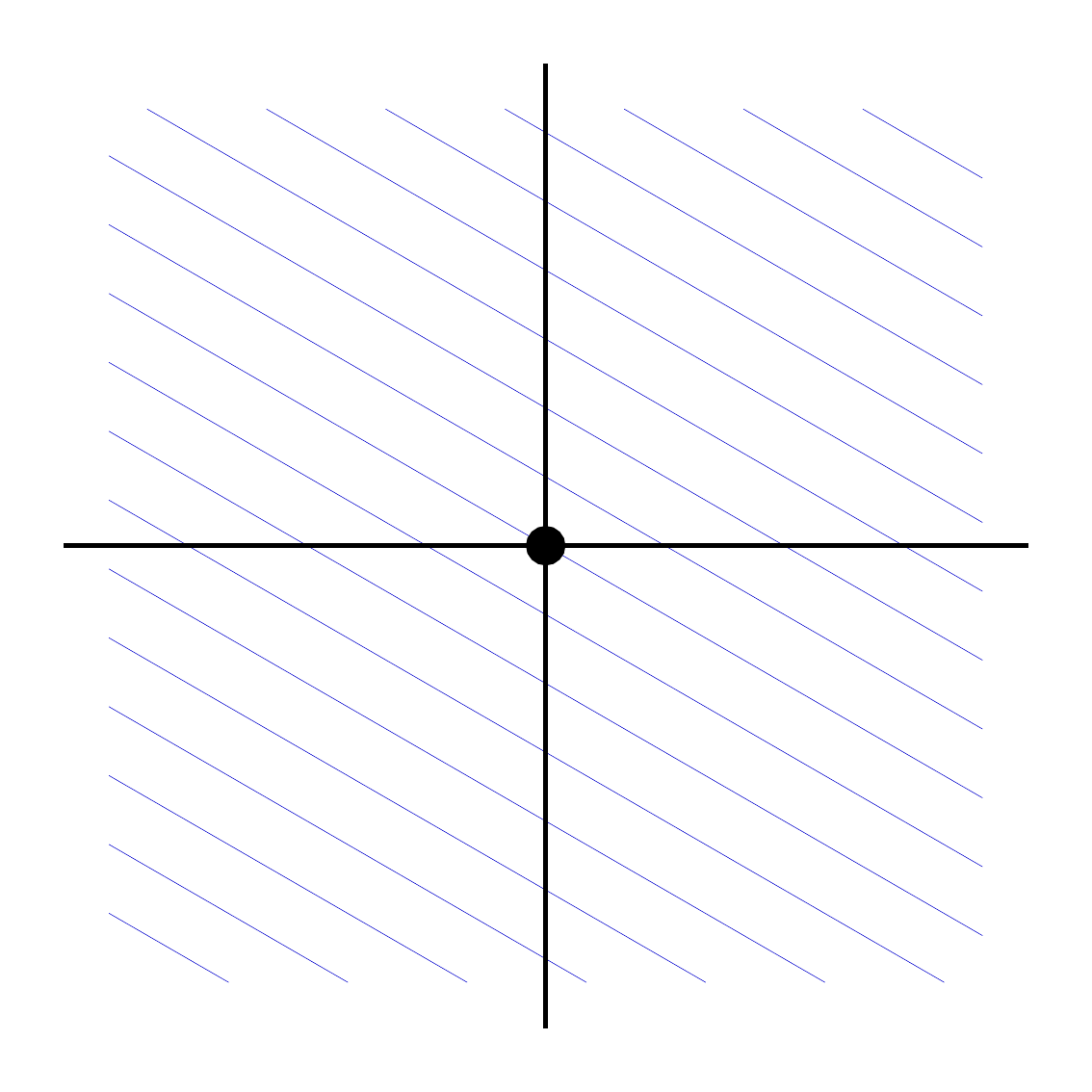}\label{fig:vert2D-a}}
\subfigure[]{\includegraphics[width=0.15\textwidth]{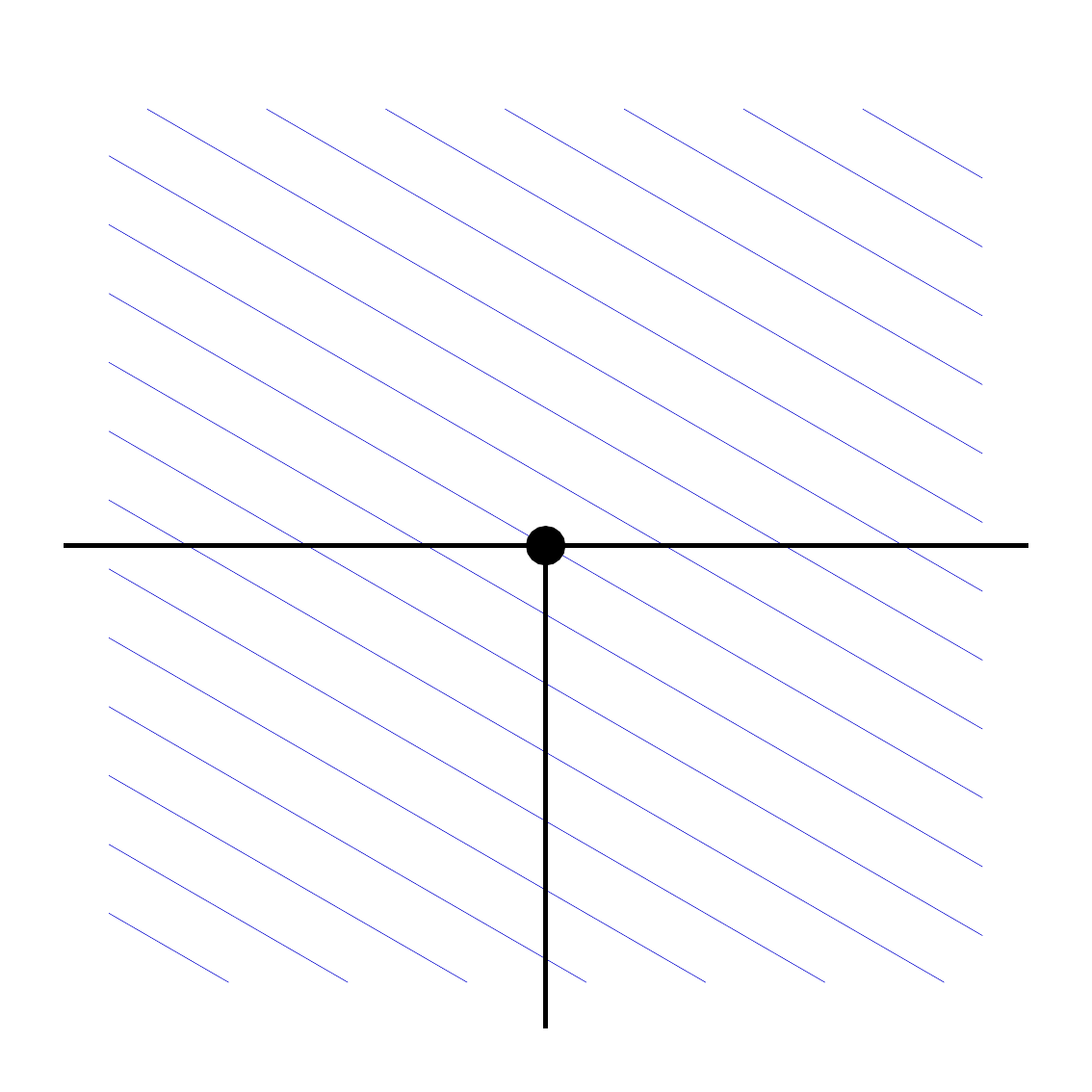}\label{fig:vert2D-b}}
\subfigure[]{\includegraphics[width=0.15\textwidth]{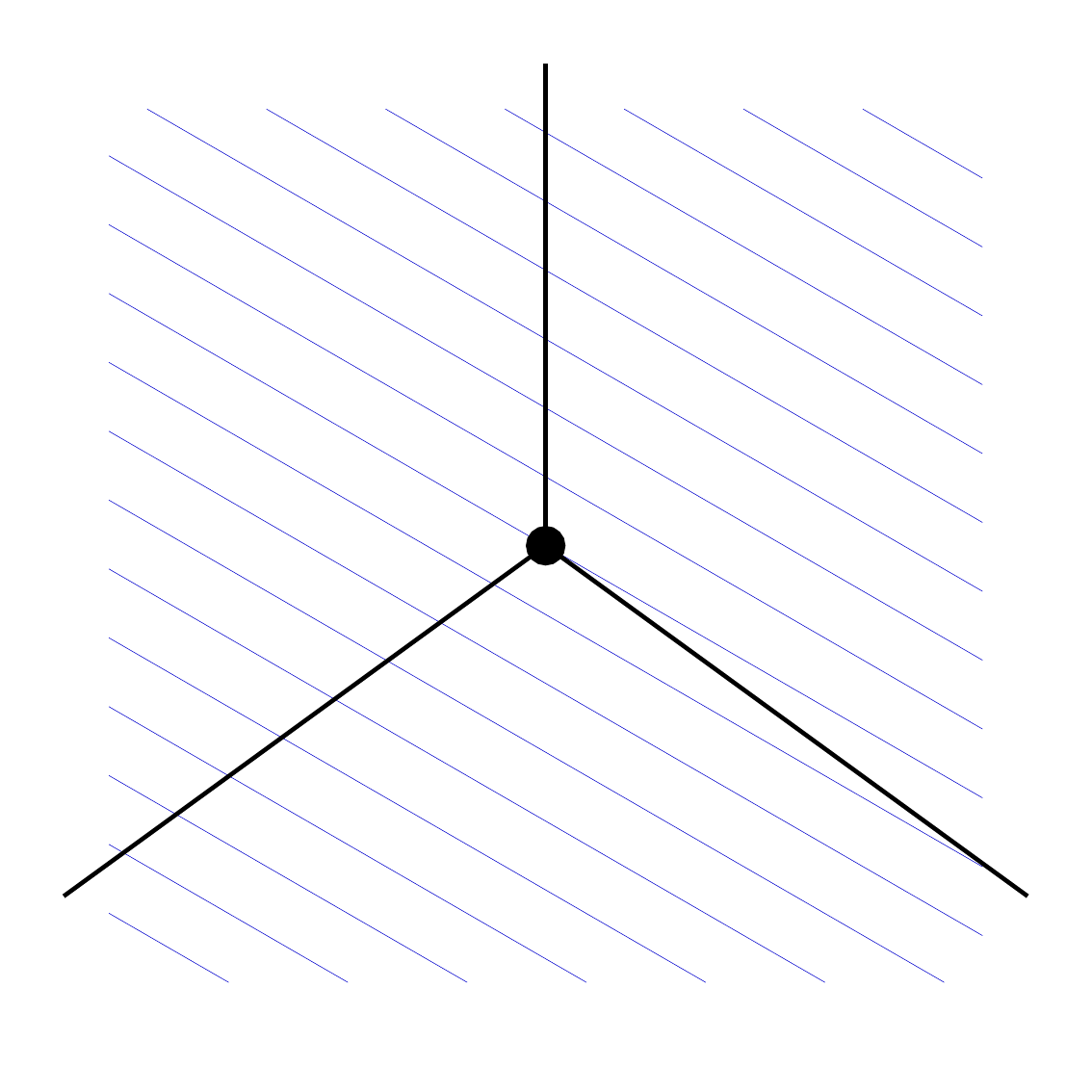}\label{fig:vert2D-c}}
\subfigure[]{\includegraphics[width=0.15\textwidth]{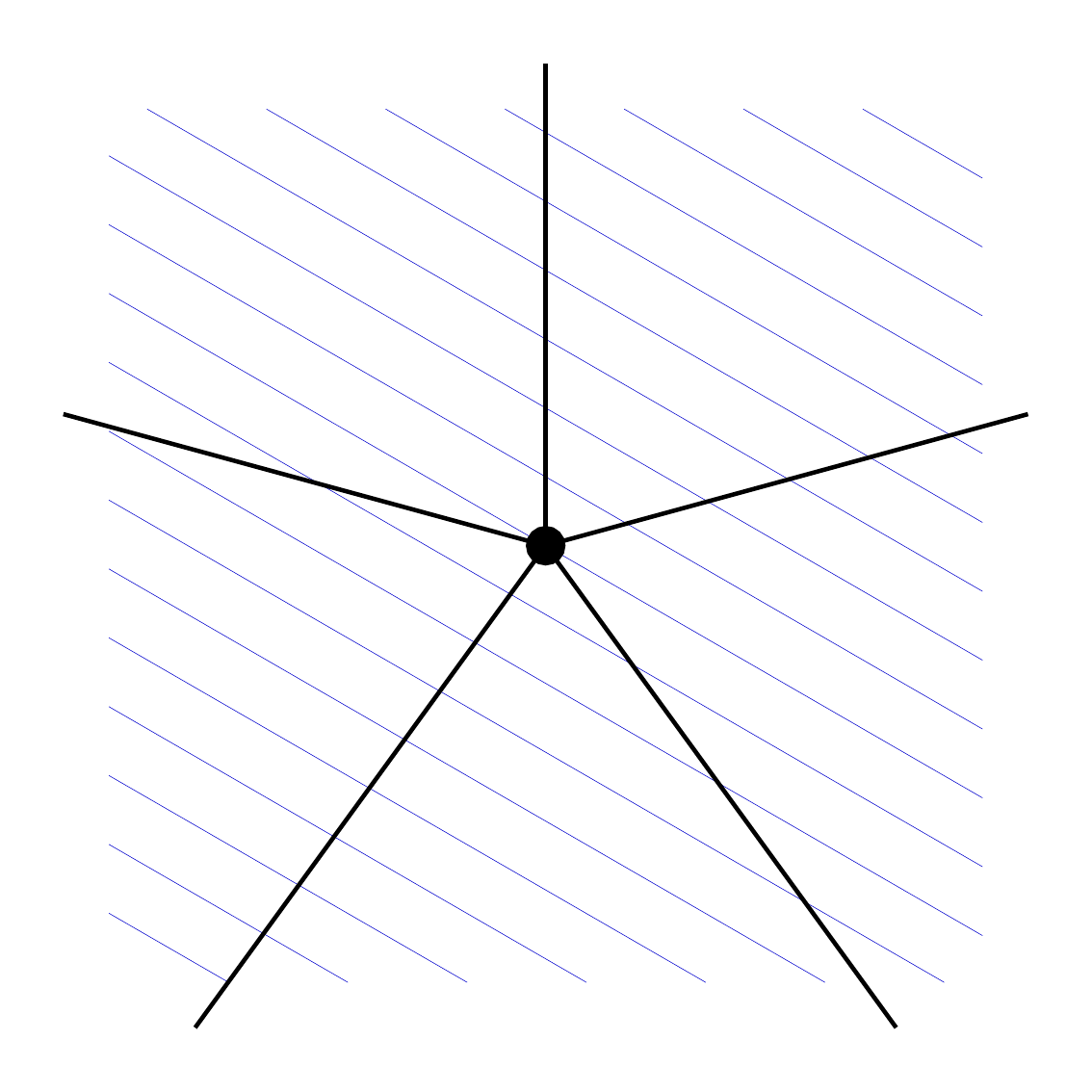}\label{fig:vert2D-d}}
    \caption{Different types of vertices in 2D ((a) regular, (b) hanging, (c) extraordinary valence 3, (d) extraordinary valence 5)}\label{fig:vert2D}
\end{figure}

For 3D domains we have three types of edges and four types of vertices. 
A 3D edge can be a \emph{regular edge}, 
a \emph{hanging edge}, or an \emph{extraordinary edge} of valence $k\neq
4$. A 3D vertex is called a \emph{regular vertex}, 
if it is shared by regular edges only; or a \emph{hanging vertex}, if it is shared by 
regular and hanging edges. 
Otherwise, the vertex is called an unstructured vertex.
The notion of unstructured vertices contains both vertices of extraordinary edges, 
so called \emph{partially unstructured vertices}, as well as 
\emph{fully unstructured vertices}. To be precise, a partially unstructured vertex is 
a vertex that is shared by exactly two extraordinary edges that can be 
covered by a single unstructured edge chart. If this is not possible, the vertex is a fully unstructured vertex. For a complete list of types of vertices and edges in 3D see Table \ref{tab:classification-3d-vertices}. 
Regular and hanging edges as well as regular and hanging vertices are covered by structured charts. 
Hence, we have considered two classes of three-dimensional
unstructured charts:  {unstructured vertex charts} (covering fully unstructured vertices) 
and {unstructured edge charts} (covering partially unstructured vertices 
and extraordinary edges). 
Unstructured edge charts are  tensor products of an unstructured vertex chart in two dimensions and an interval 
in the third dimension: there is an inner sequence of extraordinary edges 
and all other interior edges are regular or hanging. 
Similar to the two-dimensional case, an unstructured vertex chart in 3D covers 
a neighborhood of the fully unstructured vertex. 

\begin{table}[ht]
\centering
\begin{tabular}{ll}
structured edge & regular edge \\
 & hanging edge \\
unstructured edge & extraordinary edge \\
structured vertex & regular vertex \\
 & hanging vertex \\
unstructured vertex & partially unstructured vertex\\
 & fully unstructured vertex
\end{tabular}
    \caption{Classification of edges and vertices in 3D}\label{tab:classification-3d-vertices}
\end{table}

Let us consider an unstructured vertex chart $\omega_i$ composed of $k_i$ segments, 
where each segment is meshed with only one element. 
From Definition~\ref{def:structured-chart} it
follows that for any structured chart $\omega_j$ with $\omega_{i,j}\neq \emptyset$ the transition domain 
$\psi_{i,j} (\omega_{i,j}) = \omega_{j,i} $ is a box-mesh.  The same holds for
$\omega_{i,j}$, which is then  formed by one or at most  two
quadrilateral elements. There must be  (at least) $k_i$ subsets $\omega_{i,j}$, each one formed by
two adjacent elements, in order to cover $\omega_i \setminus \bxi_i$. 

\begin{figure}[!ht]
\centering
\subfigure[]{\includegraphics[height=0.18\textheight]{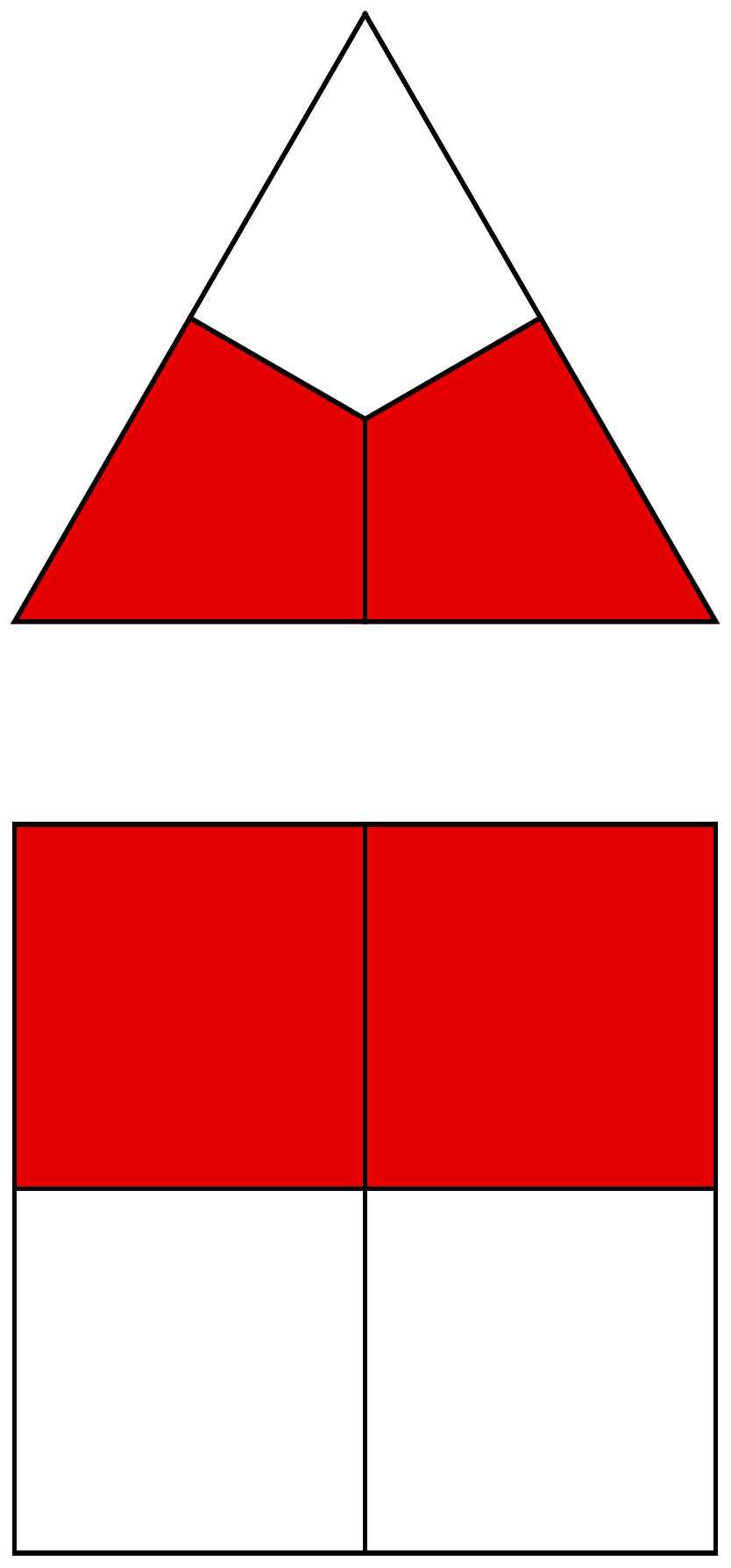}\label{fig:unstructured-charts-a}}
\hspace{0.05\textwidth}
\subfigure[]{\includegraphics[height=0.2\textheight]{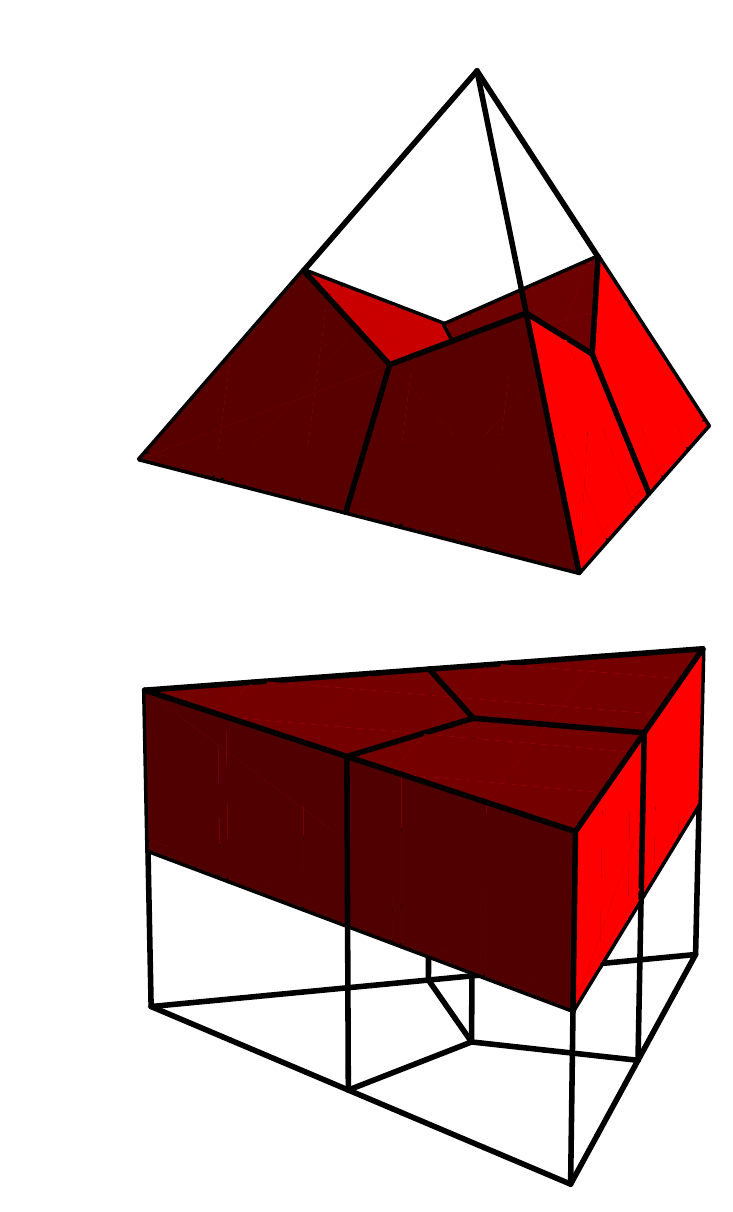}\label{fig:unstructured-charts-b}}
\hspace{0.05\textwidth}
\subfigure[]{\includegraphics[height=0.2\textheight]{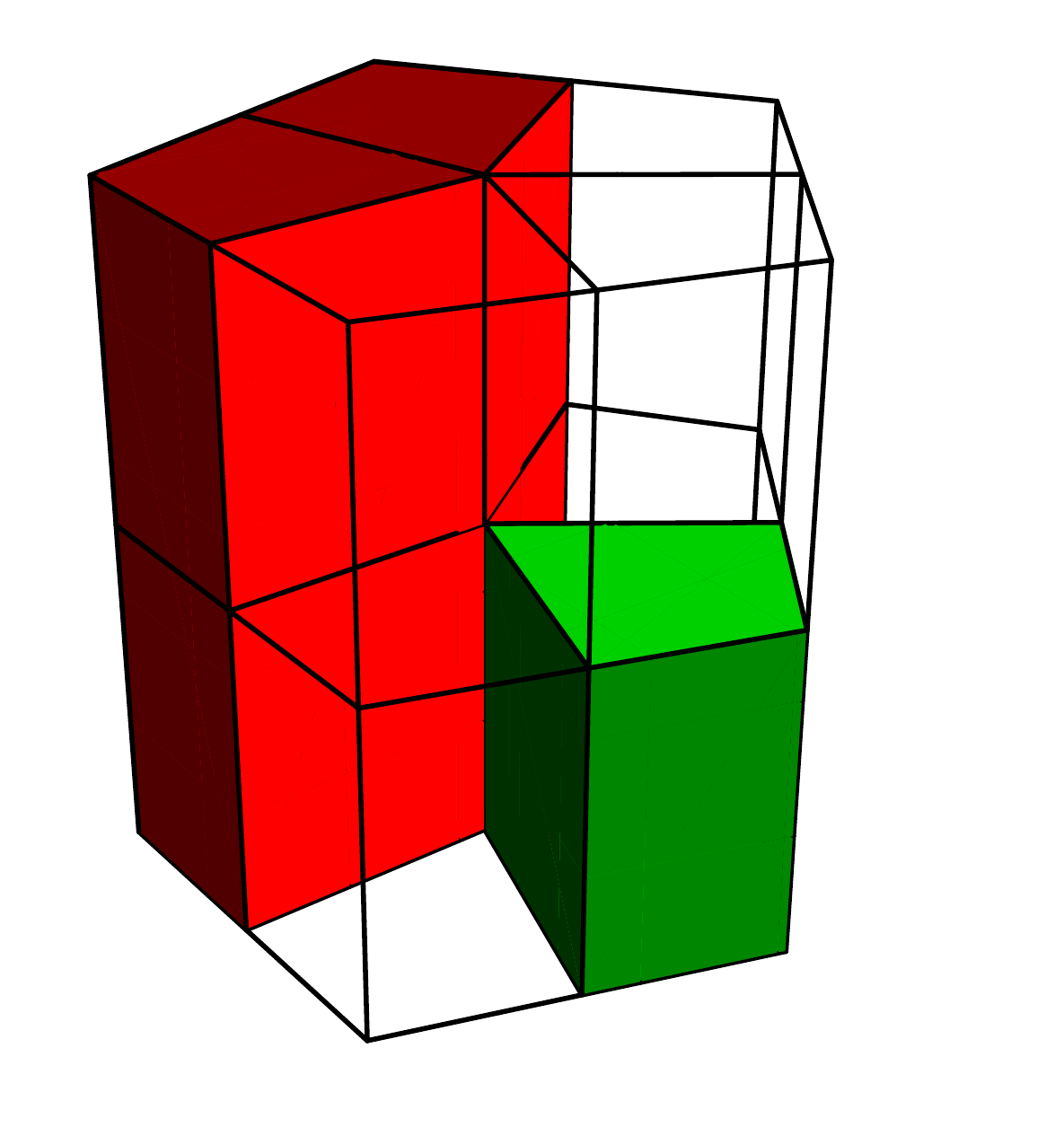}\label{fig:unstructured-charts-c}}
    \caption{Different types of unstructured charts}\label{fig:unstructured-charts}
\end{figure}

Figure \ref{fig:unstructured-charts-a} depicts a two-dimensional unstrctured vertex chart $\omega_i$ of valence 3 and a structured chart, as well as the corresponding transition domains. Since the transition domains are open, 
it can be observed easily, that in this case 3 structured charts are needed to cover $\omega_i$.
Figure \ref{fig:unstructured-charts-b} depicts a three-dimensional unstructured vertex chart $\omega_i$ of valence 4 as well as an unstructured edge chart of valence 3 and corresponding transition domains. Again, we may observe that $\omega_i$ (except for the extraordinary vertex) can be covered by the transition domains of 4 unstructured edge charts. Moreover, $\omega_i$ (except for the extraordinary vertex and edges) can be covered with 6 structured charts. 
Figure \ref{fig:unstructured-charts-c} depicts an unstructured edge chart of valence 5 and two possible 
transition domains. Note that all the transition domains are mapped box-meshes, due to the 
assumptions above. 
In all the examples presented here we assume that for each unstructured vertex chart each segment is meshed with exactly one element. This is not necessarily the case, as we presented in Figure \ref{fig:uv-charts}. Note that, by definition, the transition domains with structured charts can cover no more than 2 segments of an unstructured chart.

Note that the types of charts we consider are sufficient to represent most meshes of practical interest. 
Given an arbitrary quad- or hex-mesh without hanging vertices or edges, 
a global bisection of the mesh can be covered by structured charts, as in Definition \ref{def:structured-chart}, and unstructured charts, as in Definitions \ref{def:unstructured-chart-2d}, \ref{def:unstructured-edge-chart-3d} and \ref{def:unstructured-vertex-chart-3d}. 

This statement becomes clear, when looking at the types of vertices and edges that can occur in the bisected mesh. We consider only 3D meshes in the following. 
We show that every vertex of the bisected mesh can be covered by 
a valid chart of one of the three categories. Every vertex of the initial hex-mesh is one of the three: a structured vertex, 
a partially unstructured vertex or a fully unstructured vertex. The bisection of the initial mesh introduces new vertices 
from midpoints of edges, faces and hex-elements (one point for each edge, face and element). 
The midpoints of structured edges, faces and elements become 
structured vertices. The midpoints of unstructured edges become partially unstructured vertices. Hence no new 
fully unstructured vertices are introduced. All structured vertices can be covered by structured charts, all 
partially unstructured vertices can be covered by unstructured edge charts and all fully unstructured vertices can be covered by 
unstructured vertex charts. It is easy to see that, due to the bisection, the closure of any element can contain at most 
one extraordinary vertex. All other assumptions are trivially fulfilled.

\section{Spline manifold space on the parameter manifold}\label{sec:splinemanifoldspaces}

In this section we introduce spline spaces over the mesh $\T$ on $\Omega$, in short \emph{spline manifold spaces}.  

\subsection{General spline manifold spaces} \label{sec:general-space}

Again based on \cite{grimm1995modeling}, we define a spline space on a parameter manifold using the charts. 
This is achieved by defining proto-basis functions on the proto-mesh (Definition \ref{defi:proto-basis-functions}) and transfering 
them onto the parameter manifold using the equivalence relation induced by the transition functions (Definition \ref{defi:basis-functions-Omega}). Each chart $ \omega_i$ plays the role of a local parameter
domain. 
\begin{definition}[Proto-basis functions]\label{defi:proto-basis-functions}
We define a \emph{proto-basis} as a set 
\begin{equation}
  \label{eq:proto-splines}
  \left \{  \{b _{\bA_i}: \omega_i \rightarrow \R\}_{ \bA_i \in \A_i} \right \}_{i=1,\ldots,N},
\end{equation}
where all \emph{proto-basis functions} $b _{\bA_i}$, with $\bA_i\in\A_i$, are linearly independent functions defined on $\omega_i$. 
We further assume that for each $i=1,\ldots,N$ and $\bA_i \in \A_i$ the function $b _{\bA_i}$ fulfills 
\begin{equation}
  \label{eq:support-of-proto-basis-functions}
  \lim_{\bzeta \rightarrow \partial    \omega_i} b _{\bA_i}(\bzeta) = 0. 
\end{equation}
\end{definition}

\begin{assumption}
The proto-basis functions $b_{\bA_i}$ with $\bA_i \in \A_i$ are piecewise polynomials with 
respect to $\tau_i$, i.e. $b_{\bA_i} |_q$ is polynomial for all
$q\in\tau_i$.  
\end{assumption}

\begin{definition}[Spline manifold space]\label{defi:basis-functions-Omega}
  For each $ i=1,\ldots,N$ and $\bA_i \in \A_i $ we define $	B
  _{\bA_i} : \Omega \rightarrow \R$ such that
\begin{equation}
\begin{array}{lll}
	B _{\bA_i}|_{\Omega_i}  &=& b _{\bA_i} \circ \pi_i^{-1} \\
	B _{\bA_i}|_{\Omega \backslash \Omega_i}  &=& 0, 
\end{array}\label{eq:zero-extension}
\end{equation}
and set
\begin{equation}
  \label{eq:basis-function-set-i}
  \B_i = \{ B _{\bA_i} : \Omega \rightarrow \R, \bA_i \in \A_i \}.
\end{equation}
Furthermore, we introduce the global
index set
\begin{equation}
\label{eq:global-A}
  \A  = \left ( \bigsqcup_{i=1,\ldots,N} \A_i\right )
  \bigg/\approx 
\end{equation}
where the equivalence relation $ \approx $ is defined as follows: given $ [\bA_i,i] $ and $
[\bA_j,j]$ in $\bigsqcup_{i=1,\ldots,N}  \A_i $, then $   [\bA_i,i] \approx
[\bA_j,j] $ if and only if the two functions $B_{\bA_i} \in \B_i$ and
$ B_{\bA_j} \in \B_j$ coincide. Therefore, 
$ B _{\bA}: \Omega \rightarrow \R$ is well defined for $\bA \in \A  $, and we set 
\begin{equation}
  \label{eq:spline-global-set}
  \B= \left \{  B _{\bA} ,  \bA  \in    \A \right \} \equiv \bigcup_{i=1,\ldots,N} \B_i. 
\end{equation}
Moreover, we set
\begin{equation}
  \label{eq:extended-basis-function-set-i}
  \tilde \B_i = \{ B _{\bA}|_{\Omega_i } , \, \bA \in \tilde \A_i \},
\end{equation} 
where
\begin{equation}
  \label{eq:extended-basis-function-A-i}
  \tilde \A_i = \{ \bA \in \A : \, \text{supp}(B _{\bA}) \cap
  \text{supp}(B_{\bA_i}) \neq \emptyset \text{ for some }\bA_i \in \A_i \}.
\end{equation}
The set  $ \tilde \B_i  $ is the restriction of $\B$ onto $\Omega_i$ containing all the
functions in $\B_i$ and also  the restriction of any function whose support
intersects $\text{supp}(\B_i)$ but is not included in $\B_i$.  The functions
in $\tilde \B_i $ can be pulled back to the chart $\omega_i$.

Finally, the span of functions in \eqref{eq:spline-global-set}  is the \emph{spline manifold 
space}, the spline space on the parameter manifold $ \Omega$, denoted by  
\begin{equation}
  \label{eq:spline-space}
  \mathcal{S}= \text{span}\left \{  B _{\bA} ,  \bA  \in    \A \right
  \}. 
\end{equation}
\end{definition}
With some abuse of notation, for the global index $\bA \in \A$ we will say that $\bA \in \A_i$ if there exists an index $\bA_i \in \A_i$ such that its equivalence class through $\approx$ is equal to $\bA$.

To distinguish different types of functions, we introduce the notation
\begin{equation}\label{eq:AS-AE-AV_definition}
\begin{aligned}
	\A_s &= \bigcup_{\omega_i \text{ is  structured}} \A_i,\\  
	\A_e &= \left( \bigcup_{\omega_i \text{ is unstr. edge}} \A_i \right )\setminus \A_s,\\  
	\A_v &= \left( \bigcup_{\omega_i \text{ is unstr. vertex}} \A_i \right )\setminus \left ( \A_s \cup \A_e \right )  ,
\end{aligned}
\end{equation}
where $\A_s$ is the index set of structured functions, $\A_e$ is the index set of edge (or partially) unstructured functions and $\A_v$ is the index set of vertex (or fully) unstructured functions. It is clear that 
\begin{equation}
	\label{eq:AS-AE-AV_split}
	\A_s \cap  \A_e =\A_s \cap  \A_v =\A_e \cap  \A_v =\emptyset.
\end{equation}

 \begin{remark}
  Equation \eqref{eq:support-of-proto-basis-functions} together with \eqref{eq:zero-extension} guarantees
    that the functions are globally continuous. If discontinuous
    functions are allowed, the condition \eqref{eq:support-of-proto-basis-functions} can be omitted.
\end{remark}

\subsection{B-Spline manifold spaces}

We assume that each function that is completely supported in one chart is a function of that chart. Moreover, we assume that each structured chart is completely covered by the supports of its functions.
\begin{assumption}\label{assu:covering}
For each $i = 1,\ldots,N$ and each $\bA \in \A$ we have that if 
\begin{equation*}
	\text{supp}(B_{\bA}) \subseteq \Omega_i
\end{equation*}
then $\bA \in \A_i$. If $\omega_i$ is a structured chart as in Definition~\ref{def:structured-chart}, then 
\begin{equation*}
	\bigcup_{\bA_i \in \A_i} \text{supp}(B_{\bA_i}) = \Omega_i.
\end{equation*}
Moreover, if the support of a function $B_{\bA}$ is structured, then there exists a structured chart that covers the support.
\end{assumption}
Hence we conclude that the set $\B_i$ contains all functions that are completely supported in $\Omega_i$ and the set $\tilde\B_i$ contains all functions that have a support intersecting with $\Omega_i$.

For simplicity we assume that the functions have the same degree $p$ in each direction on each structured chart.
\begin{assumption}\label{assu:structured-chart-functions}
If $\omega_i$ is a structured chart as in Definition~\ref{def:structured-chart}, then each function $B_{\bA}$ in $ \tilde \B_i $ is  a tensor-product B-spline of degree $p$ when restricted to $\omega_i$,
i.e., there exist  (local) knot vectors $\Xi_{\bA,i,1},\ldots,\Xi_{\bA,i,d}$ such that  
\begin{equation}
  \label{eq:tp-B-spline-defn}
  B _{\bA}\circ \pi_i (\bzeta) =  b[\Xi_{\bA,i,1}](\zeta_1)
\ldots  b[\Xi_{\bA,i,d}](\zeta_d), \qquad \forall  \bA \in \tilde \A_i,\forall \bzeta \in \omega_i,
\end{equation}
where $  b[\Xi](\zeta)$ is the univariate B-spline with local knot
vector $\Xi$. In this notation the degree $p$ of the B-spline is given implicitly by the length of the local knot vector.
\end{assumption}

Finally, the previous assumption is extended to unstructured edge charts, since they behave like structured charts along the third parametric direction, and like unstructured vertex charts along the first two parametric directions.
\begin{assumption}\label{assu:unstructured-edge-chart-3d}
If $\omega_i$ is an unstructured edge chart as in Definition \ref{def:unstructured-edge-chart-3d}, then each function $B_{\bA}$ in $ \tilde \B_i $ is a product of a bivariate function and a B-spline of degree $p$ in $\zeta_3$ when restricted to $\omega_i$, i.e. there exist a function $\beta_{\bA,i}$ and a local knot vector $\Xi_{\bA,i,3}$ such that
\begin{equation}\label{eq:edge-chart-B-spline-defn} 
	B _{\bA}\circ \pi_i (\bzeta) = \beta_{\bA,i}(\zeta_1,\zeta_2) b[\Xi_{\bA,i,3}](\zeta_3), \qquad \forall  \bA \in \tilde \A_i,\forall \bzeta \in \omega_i.
\end{equation}
\end{assumption}

\begin{remark}\label{rem:extraordinary-functions-defined-on-structured-charts}
From the definition of the unstructured charts $\omega_j$, and under Assumption~\ref{assu:structured-chart-functions}, we have that the  functions   $B
  _{\bA}$, with $\bA \in \A_j$,  are fully  defined by their 
restrictions to the transition domains $\omega_{j,i}$ with structured charts. On each $\omega_{j,i}$ the functions in $\B_j \cap \tilde\B_i$
are equivalent to (mapped) tensor-product B-splines according to \eqref{eq:tp-B-spline-defn}.
\end{remark}

\subsection{Isogeometric function spaces using spline manifolds}

We consider a domain $\Sigma \subset \R^n$ which can be 
interpreted as a $d$-dimensional manifold with $d\leq n$. The most interesting cases, and the ones we focus on, are 
a two-dimensional planar domain $\Sigma \subset
\R^2$, a surface $\Sigma \subset \R^3$ or a three-dimensional volumetric domain $\Sigma \subset
\R^3$.

\begin{definition}\label{defi:G}
The physical domain $\Sigma$ is given by a \emph{spline manifold parametrization} ${\bf G} \in ({\cal S})^n$, with 
\begin{equation}\label{eq:G}
 \f G : \Omega \rightarrow \Sigma.
\end{equation}
\end{definition}
In practical situations, this parametrization is defined by associating a control point to each function in ${\cal B}$. Notice that the geometry $\Sigma$ inherits the manifold structure of the parameter manifold $\Omega$.
Indeed, the parametrization $\f G$ can be considered as a piecewise defined function, where 
$ \f G_i :  \Omega_i \rightarrow \Sigma_i$ is a tensor-product B-spline parametrization for every structured chart $\Omega_i$. Then we have 
\begin{displaymath}
  \Sigma = \bigcup_{i=1}^{N} \Sigma_{i},
\end{displaymath}
where the charts $\Sigma_i$ form an \emph{atlas} of $\Sigma$.

Then, the isogeometric function space over the manifold $\Sigma$ is given as follows.
\begin{definition}
Given a \emph{spline manifold parametrization} $\f G\in (\S)^n$ as in Definition~\ref{defi:G}, we define on the manifold $\Sigma$ the \emph{isogeometric function space} as 
\begin{equation}
  \label{eq:isogeometric_space}
  \V = \{ f:\Sigma \rightarrow \R, \, \text{  such that }\, f = \hat f \circ \f G^{-1} \text{ for } \hat f \in
  \S \}.
\end{equation}
The isogeometric space is well-defined, if the geometry parametrization is invertible. {We could move here the assumption on $det \nabla \bG^T \nabla \bG$.}
\end{definition}

Similar to the spline manifold space $\S$ itself, the isogeometric space $\V$ can be interpreted as a piecewise defined function space. Indeed, each function in ${\cal B}_i$ can be composed with $\f G^{-1}$ to define the corresponding function in $\Sigma$, with its support contained in the chart $\Sigma_i$.

\subsection{Relation to existing constructions}

We note that several constructions of unstructured spline spaces existing in the literature fit in the framework of B-spline manifolds, in some cases with minor modifications. For instance, the definition of unstructured T-splines as presented in \cite{wang2011converting,wang2012converting} for quad-meshes and hex-meshes, the $G^1$-continuous unstructured T-splines as presented in \cite{Scott2013}, as well as multi-patch B-splines with enhanced smoothness as developed in \cite{Buchegger2015}. The development of an abstract framework for the construction of these spaces can serve as the starting point for a deeper mathematical analysis of their properties. In the following we detail how these three particular examples fit into the framework of B-spline manifolds. In all three cases, the idea is to split the set of basis functions into structured and unstructured basis functions, and introduce a set of charts covering the whole mesh. In this context, a function is called structured, if its support is covered by a structured mesh (possibly with hanging nodes). Otherwise it is called unstructured. 

\paragraph{T-splines over unstructured meshes \cite{wang2011converting}}
    In this configuration the mesh $\T$ we 
  consider on $\Omega$ is the  B\'ezier mesh and not the T-mesh. 
  Then, we can define the set of charts by taking the support 
  of each structured function as a single structured chart, 
  while  the support of each unstructured function is taken 
  as one unstructured vertex chart. In the construction of 
  \cite{wang2011converting} it may happen that one element contains
  two extraordinary vertices. In this case two unstructured 
  vertex charts overlap, violating Definition~\ref{def:unstructured-chart-2d} (C). 
  This condition is a technicality which simplifies  the mathematical 
  framework, and could be removed. It can also be fulfilled with one 
  level of refinement of the T-mesh, as we explained in 
  Section~\ref{sec:meshing}. In a similar way, the 
  constructions for trivariate functions as presented in \cite{wang2012converting} fit, with some 
  technical  restrictions, into the present framework.

\paragraph{$G^1$-smooth T-splines over unstructured meshes \cite{Scott2013}}
 This construction is similar to the construction in \cite{wang2011converting}. The support of each structured function can again be interpreted as a single structured chart. For each extraordinary vertex we can define one sufficiently large unstructured chart $\omega_j$, such that it covers the support of all functions that are non-zero at the extraordinary vertex. In this case, the unstructured functions do not fulfill Assumption \ref{assu:structured-chart-functions}, since their degree is increased in the vicinity of the extraordinary vertex, and their restriction to a structured chart is not a B-spline basis function as in \eqref{eq:tp-B-spline-defn}, but a suitable linear combination of B-splines of higher degree.   

\paragraph{Multi-patch B-splines with enhanced smoothness \cite{Buchegger2015}}
 The authors propose a construction based on a multi-patch representation of the domain. This can be interpreted as a manifold structure with closed charts (patches), that intersect only at the boundary. 
However, since there is a one-to-one correspondence of parameter directions along any shared boundary between two patches, the multi-patch representation can be transformed naturally to a parameter manifold representation by simply enlarging the patches along the parameter direction crossing the boundary to obtain open charts. Note that in general it is necessary to define more than one chart to cover one patch. Again, there are unstructured functions at the extraordinary vertices, and unstructured vertex charts have to be introduced to cover their support.

\section{Analysis-suitable spline manifold spaces}
\label{sec:analysis-suitable}

We introduce in this section the conditions for the construction of analysis-suitable B-spline spaces on manifolds, that is, spaces that have good properties for the solution of differential problems. The key tool for this construction is the definition of a (stable) dual basis, which is a set of functionals $\{\Lambda_\bA, \bA \in \A \}$ such that
\begin{equation*}
\Lambda_\bA(B_{\bA'}) = \delta_{\bA \bA'}, \quad \forall \bA, \bA' \in \A,
\end{equation*}
where $\delta_{\bA \bA'}$ represents the Kronecker delta. A condition for the construction of a dual basis for structured T-splines was given in \cite{beirao2012analysis,beirao2013analysis}, under the name of \emph{dual-compatibility}.

We present the construction of dual functionals on the parameter manifold~$\Omega$ in Section~\ref{sec:dual-functionals-general}, starting from a proto-dual basis on each chart, and then following the scheme of spline manifold spaces introduced in Section~\ref{sec:general-space}. To guarantee that these dual functionals form a dual basis we need to add some conditions to the spline manifold. In Section~\ref{sec:dual-compatible-condition} we give a \emph{dual-compatibility} condition for spline manifolds, which generalizes the condition in \cite{beirao2012analysis,beirao2013analysis} to the unstructured setting. In this configuration, a global dual basis can be derived from the proto-dual bases defined on the charts. Then, in Section~\ref{sec:dual-basis-construction} we present an explicit configuration, with a specific construction of the proto-basis functions and the proto-dual basis on
unstructured vertex and edge charts. In this configuration, which is only $C^0$-continuous at extraordinary features, the dual functionals in~$\Omega$ can be derived from the proto-dual basis without any further modification, which also guarantees the stability of the dual functionals.
Finally, in Section~\ref{sec:properties-dual-compatible} we show the
typical application of a dual basis: we prove otpimal approximation  properties
of the isogeometric space on  a simple but interesting example configuration.

\subsection{The dual basis for spline manifold spaces} \label{sec:dual-functionals-general}

In the following we use this well-known result.
\begin{proposition}\label{proposition:trivial-dual-basis}
  Given a set of finitely many $L^2$  functions $\{ b_\alpha\} $ that are linearly
  independent, there exist  functionals $\lambda_\alpha:L^2
  \rightarrow \R$ that are dual to the functions, i.e., $\lambda_\alpha(b_{\alpha'})=
  \delta_{\alpha \alpha'}$.
\end{proposition}

As for the definition of the basis functions for spline manifold spaces introduced in Section~\ref{sec:general-space}, the starting point is a set of dual-functionals on each chart, that exist thanks to the previous proposition and Definition~\ref{defi:proto-basis-functions}.
\begin{definition}[Proto-dual basis] \label{def:proto-dual}
We define a \emph{proto-dual basis} as a set
\begin{equation*}
\{ \{ \hat\lambda_{\bA_i}: L^2(\omega_i) \rightarrow \R \}_{\bA_i \in \A_i} \}_{i=1, \ldots, N}, 
\end{equation*}
where the functionals $\hat\lambda_{\bA_i}$ form a dual basis for the proto-basis functions in the chart, that is
\begin{equation*}
\hat\lambda_{\bA_i}(b_{\bA'_i}) = \delta_{\bA_i \bA'_i}, \quad \forall \bA_i, \bA_i' \in \A_i.
\end{equation*}
\end{definition}
Note that the existence of a proto-dual basis is guaranteed by Proposition \ref{proposition:trivial-dual-basis}, since the proto-basis functions $b_{\bA_i}$ are linearly independent. We will assume that the proto-dual functionals satisfy the following.

\begin{assumption}\label{ass:dual-functionals}
For any indices $[\bA_i,i] \approx [\bA_j,j]$ that belong
to the same equivalence class $\bA \in \A$, defined as
in~\eqref{eq:global-A}, it holds that  $\hat\lambda_{\bA_i} (\phi\circ \pi_i) =
\hat\lambda_{\bA_j} (\phi\circ \pi_j), \quad \forall \phi \in L^2(\Omega)$.
\end{assumption}

The assumption that two equivalent indices are associated to the same
dual functional is natural, since by~\eqref{eq:global-A} they are also
associated to the same basis function, and it allows the following global
definition.

\begin{definition}[Manifold functionals] \label{def:dual-functionals} For any $\bA \in \A$ we define the functional
\begin{equation} \label{eq:dual-functionals}
\hat\Lambda_\bA(\phi) = \hat\lambda_{\bA_i} (\phi\circ \pi_i), \quad \forall \phi \in L^2(\Omega),
\end{equation}
where $[\bA_i,i]$ is one instance of the equivalence class $\bA$.
\end{definition}
From \eqref{eq:dual-functionals},  the support of the dual functional
$\hat \Lambda_{\bA}$ is contained  in  $\Omega_i$. The most used dual
bases for splines, such as the one by Schumaker \cite{Schumi} and the
ones by Lee, Lyche and M{\o}rken \cite{LLM01} satisfy  the stronger condition
that  the support of $\hat \Lambda_{\bA}$  is the same of the corresponding function $B_{\bA}$.

Note that in general the dual functionals $\{\hat\Lambda_\bA, \bA \in \A\}$ do not form a dual basis for ${\cal B} = \{B_{\bA}, \bA \in \A\}$. 
In the following section we introduce a new condition that ensures that a dual basis can be derived by a modification of the dual functionals $\hat \Lambda_\bA$.

\subsection{Dual-compatible spline manifold spaces}
\label{sec:dual-compatible-condition}

For the definition of the dual-compatibility condition we follow and
extend the recent review paper \cite{beirao2014actanumerica}. 
Two knot vectors $\Xi' =\{ \xi'_1, \ldots,
\xi'_{p+2}\}$ and $\Xi'' =\{\xi''_1, \ldots, \xi''_{p+2}\}$ \emph{overlap}  if there exists a  knot vector  $\Xi =\{ \xi_1, \ldots,
\xi_{k}\}$ and two integers  $k'$ and $k''$ such that $\xi'_i = \xi_{i+k'}$ and $
\xi''_i = \xi_{i+k''}$, for  $i = 1, \ldots, p+2$. 
We generalize the dual-compatibility condition to spline manifold
spaces in the following definition. 
\begin{definition}[Dual-compatible spline manifold spaces]\label{def:dual-compatible}
Under the previously stated assumptions, 
the set $\B$ defined in  \eqref{eq:spline-global-set} is
 dual-compatible if the following conditions hold
\begin{enumerate}
\item  for all $ \bA \in \A_i \subset  \A_s $ and $ \forall \bA' \in
  \widetilde \A_i $  there exists an index $l \in \{1,\ldots,d\}$, such that the  knot vectors $\Xi_{\bA,i,l}$ and $\Xi_{\bA',i,l}$ as in
  \eqref{eq:tp-B-spline-defn} are different and overlap;
\item  for all $  \bA \in \A_i \cap  \A_e $, with $\omega_i$ being an unstructured edge chart, and for all $ \bA' \in
  \A_i$, either  the knot vectors $\Xi_{\bA,i,3}$ and $\Xi_{\bA',i,3}$
  as in \eqref{eq:edge-chart-B-spline-defn} overlap and are different,
  or $\beta_{\bA,i} \neq \beta_{\bA',i}$;
\item  for all $  \bA \in \A_i \cap  \A_e $, with $\omega_i$ being an unstructured edge chart, and for all $ \bA' \in
  (\widetilde \A_i \setminus \A_i ) \cap  (\A_e  \cup  \A_v  )$ the knot vectors $\Xi_{\bA,i,3}$ and $\Xi_{\bA',i,3}$
  as in \eqref{eq:edge-chart-B-spline-defn} overlap and are different.
\end{enumerate}
Here $\A_s$, $\A_e$ and $\A_v$ are defined as in \eqref{eq:AS-AE-AV_definition}.
\end{definition}

\cite{beirao2014actanumerica} (as well as the previous papers
\cite{beirao2012analysis,beirao2013analysis}) the authors deal only with the
structured case, and actually the dual-compatibility condition in 
\cite{beirao2014actanumerica} corresponds to point~1 in 
Definition~\ref{def:dual-compatible}. The new definition extends the 
dual-compatibility condition to an unstructured configuration. 
As in \cite{beirao2014actanumerica},  Definition
\ref{def:dual-compatible} gives a sufficient condition for 
the existence of a dual basis. For that, we use two technical
ingredients: the first is a univariate $L^2$-stable dual functional $\lambda [\Xi]$, such that, if $\Xi$ and $\Xi'$ are overlapping, 
\begin{equation}\label{eq:1D-dual-fun}
 \lambda [\Xi] (b [\Xi']) = \left \{
    \begin{aligned}
      & 1 \text{ if }\Xi = \Xi',\\
      & 0 \text{ if } \Xi \neq \Xi';
    \end{aligned}
\right  .
\end{equation}
see, for example, \cite{Schumi,LLM01}; the second is the existence of a proto-dual basis on each chart, guaranteed by Proposition~\ref{proposition:trivial-dual-basis}.
Our main result follows.
\begin{theorem}\label{thm:linear-independence}
If the set $\B$ is dual-compatible, then there exists a set of functionals 
  \begin{equation*}
  \B^*= \left \{  \Lambda _{\bA} : L^2(\Omega) \rightarrow
\mathbb{R} \text{ such that }\bA  \in    \A \right \},
\end{equation*}
that is dual to  $\B$, i.e., 
\begin{equation} \label{eq:orthogonal-all}
\forall \bA,\bA' \in \A, \,\Lambda_{\bA}(B_{\bA'}) = \delta_{\bA \bA'}.
\end{equation}
\end{theorem}
\begin{proof}
First, we  construct dual functionals for the different types of charts.
Let $\omega_i$ be a structured chart. Given  $\bA\in \A_i\subset \A_s$ we set
\begin{equation}
  \label{eq:tp-Lambda-defn}
 \hat\lambda_{\bA}=  \lambda [\Xi_{\bA,i,1}]\otimes \ldots \otimes
 \lambda[\Xi_{\bA,i,d}],
\end{equation}
where the knot vectors $\Xi_{\bA,i,k}$ are given as in \eqref{eq:tp-B-spline-defn}. 
Let $\omega_i$ be an unstructured edge chart. 
Given  $\bA\in\A_i\cap \A_e$, let $\A_{i,\bA} \subset \A_i\cap \A_e$ be
the set of indices $\bA'$ such that $\Xi_{\bA,i,3} = \Xi_{\bA',i,3}$
as in \eqref{eq:edge-chart-B-spline-defn}. Since the set $\B_i$ is
linearly independent, the same holds for the set
\begin{equation}
  \label{eq:section}
  \{\beta_{\bA',i}: \;
\bA'\in\A_{i,\bA}\}
\end{equation}
and by Proposition
\ref{proposition:trivial-dual-basis}  there exist functionals $\{\hat
\lambda_{\bA',i}: \; \bA'\in\A_{i,\bA}\}$ that are dual to
\eqref{eq:section}.
 We define the proto-dual functional $\hat
 \lambda_{\bA}:\omega_i\rightarrow \R$
\begin{equation*} \label{eq:edge-Lambda-defn}
 \hat \lambda_{\bA}=  \hat \lambda_{\bA,i}\otimes 
 \lambda[\Xi_{\bA,i,3}],
\end{equation*}
Finally, let $\omega_i$ be an unstructured vertex chart. Consider $\A_i\cap \A_v$, and again by Proposition
\ref{proposition:trivial-dual-basis}   the set $\B_i$, which is assumed to be linearly independent, admits a  proto-dual basis  $\{\hat \lambda_{\bA}: \bA \in \A_i\cap \A_v\}$.
Having defined $\lambda_{\bA}$ for all $\bA \in \A$, we construct
the corresponding $\hat \Lambda_\bA$  as in
\eqref{eq:dual-functionals}.  

It is easy to check that, by the construction above and Definition
\ref{def:dual-compatible}, it holds
\begin{equation}
  \label{eq:As-Lambda-duality}
  \forall \bA \in \A_s,  \forall \bA' \in (\A_s \cup \A_e \cup \A_v) ,
  \quad \hat \Lambda_\bA (B_{\bA'}) = \delta_{\bA\bA'},
\end{equation}
\begin{equation}
  \label{eq:Ae-Lambda-duality}
  \forall \bA \in \A_e,  \forall \bA' \in ( \A_e \cup \A_v) ,
  \quad \hat \Lambda_\bA (B_{\bA'}) = \delta_{\bA\bA'},
\end{equation}
\begin{equation}
  \label{eq:Av-Lambda-duality}
  \forall \bA \in \A_v,  \forall \bA' \in \A_v,
  \quad \hat  \Lambda_\bA (B_{\bA'}) = \delta_{\bA\bA'}.
\end{equation}
As already noted, from 
Definition~\ref{def:dual-functionals} we also infer
\begin{equation}
  \label{eq:split-support}
  \forall \bA \in \A, \forall  \bA' \in \A \setminus \tilde \A_i ,
  \quad \hat  \Lambda_\bA (B_{\bA'}) = 0.
\end{equation}
  
In general, the set of functionals $\{\hat \Lambda_\bA \}_{ \bA \in \A}$ is not dual
to the set of functions $\B=\{B_\bA \}_{ \bA \in \A}$, but we can
easily fix it by defining, for all $\phi\in L^2(\Omega)$, the functionals:
\begin{equation}
  \label{eq:Lambda-from-hatLambda}
  \begin{aligned}
    \forall \bA &\in \A_s,  &\Lambda_{\bA}(\phi) &=  \hat
    \Lambda_{\bA}(\phi), \\
  \forall \bA &\in \A_e,  & \Lambda_{\bA}(\phi)&=  \hat
  \Lambda_{\bA}(\phi) - \sum_{\bA'' \in (\tilde\A_i \backslash \A_i)
    \cap \A_s} \hat \Lambda_{\bA}(B_{\bA''}) \Lambda_{\bA''}(\phi),\\
 \forall \bA &\in \A_v,  &\Lambda_{\bA}(\phi) &=  \hat
 \Lambda_{\bA}(\phi) - \sum_{\bA'' \in \tilde\A_i  \backslash \A_i} \hat \Lambda_{\bA}(B_{\bA''}) \Lambda_{\bA''}(\phi).
  \end{aligned}
\end{equation}
By going through all combinations for $\bA,\bA' \in \A_s,\A_e, \A_v$ and using \eqref{eq:As-Lambda-duality}--\eqref{eq:Av-Lambda-duality}
as well as the condition \eqref{eq:split-support}, it follows that the set $\{\Lambda_\bA \}_{ \bA \in \A}$ as defined in 
\eqref{eq:Lambda-from-hatLambda} fulfills  \eqref{eq:orthogonal-all}.
\end{proof}

\subsection{Definition of a basis and corresponding dual basis for a simplified configuration}
\label{sec:dual-basis-construction}

The dual compatibility condition of the previous section ensures the
existence of a dual basis, which implies the linear independence of
the basis functions in~${\cal B}$.
We now focus on a simplified configuration  of B-spline manifold spaces,
which allows for exactly one extraordinary function in every
unstructured vertex chart. We explicitly construct a  dual basis within the abstract
  framework of Section \ref{sec:dual-compatible-condition}. However,  
  taking advantage of the specific case considered here, we obtain dual functionals
  whose support is the same as the support of the corresponding functions.

In Assumption \ref{assu:extraordinary-function} we restrict our studies to more simple 
unstructured vertex charts, where each segment corresponds to an element of the mesh. Moreover, we give an explicit representation for the extraordinary vertex functions. Note that one may also define a collection of extraordinary vertex functions that are discontinuous at all or some of the element boundaries within the unstructured vertex chart. For simplicity, we consider only one $C^0$ continuous extraordinary vertex function for each unstructured vertex chart.

\begin{assumption}\label{assu:order-for-DC}\label{assu:extraordinary-function}
For each unstructured vertex chart $\omega_i$,  we
assume that the mesh $\tau_i$ is the union of the meshes $\sigma_{i,\ell}$ on the segments $s_{i,\ell}$, where each $\sigma_{i,\ell}$ contains a single element $q_{i,\ell}$ which covers the whole segment. 
Moreover, we assume that there exists exactly one index $\bA \in \A_i$ and the function $B_{\bA} \in \B_i$ is named \emph{unstructured vertex function}. The extraordinary vertex function is continuous, has the value 1 at the extraordinary vertex, and for each element $q_{i,\ell}\in \tau_i$ there exists a $d$-linear mapping $\hat\psi_{i,\ell} : q_{i,\ell} \rightarrow \left]0,1\right[^d$, such that 
\begin{equation}\label{eq:unstr-B-spline-dual-compatible}
 B _{\bA}\circ \pi_i (\bzeta) = \hat b_d \circ \hat\psi_{i,\ell}(\bzeta), \mbox{ for all }\bzeta\in q_{i,\ell},
\end{equation}
where $\hat b_d $ is defined as in
\begin{eqnarray}
  \hat b_d(\zeta_1,\ldots,\zeta_d)&=& (1-\zeta_1)^p\ldots (1-\zeta_d)^p,
  \nonumber\\
  &=& b[\Xi_0](\zeta_1)\ldots b[\Xi_0](\zeta_d),
  \quad \forall (\zeta_1,\ldots,\zeta_d) \in \left]0,1\right[^d,\label{eq:b-hat}
\end{eqnarray}
with $\Xi_0 = [0,\ldots,0,1]$.

For each unstructured edge chart $\omega_i = \twoD{\omega}_i \times \oneD{\omega}_i $, the mesh of the underlying bivariate chart $\twoD{\omega}_i$ is again given such that each segment $\twoD s_{i,\ell}$ of the bivariate chart is covered by a single element $\twoD{q}_{i,\ell}$, 
the one-dimensional structured chart $\oneD{\omega}_i$ is partitioned
into a mesh $\oneD\tau_i = \{ \oneD q_1, \ldots , \oneD q_{m_i} \} $, and the mesh $\tau_i$ on $\omega_i$ is given via
\begin{equation*}
	\tau_i = \{ \twoD q_{i,\ell} \times \oneD q_k :\; \twoD q_{i,\ell} \in \twoD\tau_i,\oneD q_k \in \oneD\tau_i \}.
\end{equation*}
Moreover, we assume that for all $\bA',\bA'' \in \A_i$ the \emph{unstructured edge functions} $B_{\bA'},B_{\bA''}$ are given as in \eqref{eq:edge-chart-B-spline-defn}, where $\beta_{\bA',i} = \beta_{\bA'',i}$ are equal to the same unstructured vertex function of dimension $d=2$. Hence, for each segment $\twoD{q}_{i,\ell}\times \left] a_{i,3},b_{i,3} \right[$ there exists a bilinear mapping $\hat\psi_{i,\ell} : \twoD{q}_{i,\ell} \rightarrow \left]0,1\right[^2 $, such that 
\begin{equation*}
	B _{\bA}\circ \pi_i (\bzeta) = \left(\hat b_2 \circ \hat\psi_{i,\ell } (\zeta_1,\zeta_2)\right) \; b[\Xi_{\bA,i,3}](\zeta_3), \qquad \mbox{for all } \bzeta \in \twoD{q}_{i,\ell}\times \left] a_{i,3},b_{i,3} \right[,
\end{equation*}
corresponding to the representation in \eqref{eq:edge-chart-B-spline-defn}.
\end{assumption}
Given an unstructured vertex chart $\omega_i$ this means that, denoting by $j(q)$ a chart index such that $\omega_{j(q)}$ is a structured chart and $q
  \subset \omega_{i,j(q)}$, there exist knot vectors
  $\Xi_{\bA,q,j(q),1}, \ldots, \Xi_{\bA,q,j(q),d}$   of the kind 
  \begin{equation}
    \label{eq:local-knot-vectors-for-extraordinary-function}
  [ \underset{p+1\text{
        times}}{\underbrace{\xi',\ldots,\xi' }, \xi''} ] \text{ or } [ \xi',\underset{p+1\text{
        times}}{\underbrace{\xi'',\ldots,\xi'' }} ] 
  \end{equation}
such that  
\begin{equation*}
 B _{\bA}\circ \pi_j (\bzeta) =  b[\Xi_{\bA,q,j(q),1}](\zeta_1) \cdot
 \ldots \cdot  b[\Xi_{\bA,q,j(q),d}](\zeta_d), 
\end{equation*}
for all $\bzeta$ such that $ \psi_{j(q),i} (\bzeta) \in q$. We can now define a dual basis explicitly, using this representation. A similar representation can be derived for unstructured edge charts.

In the following we define proto-dual functionals $\lambda_\bA$ and corresponding manifold functionals $\Lambda_\bA$ as in Definition~\ref{def:dual-functionals} which fulfill 
\begin{equation} \label{eq:orthogonal-tilde}
\Lambda_{\bA_i}(B_{\bA_i'}) = \delta_{\bA_i \bA_i'}, \; \forall \bA_i \in \A_i, \text{ and } \forall \bA_i' \in \tilde \A_i.
\end{equation}
Hence they form a dual basis for $B_{\bA}$, with $\bA\in\A$. For each
structured function (a function completely supported by a structured
chart), the dual functional is given as the tensor-product of 
univariate functionals \eqref{eq:1D-dual-fun}. For an unstructured vertex function we construct a dual functional by adding the contributions of each segment of the unstructured vertex chart, each being a structured subdomain. For unstructured edge functions, we construct a dual functional analogously to the basis function in 
\eqref{eq:edge-chart-B-spline-defn}, that is, the tensor-product of the dual functional of a bivariate unstructured vertex function with the dual functional of a univariate B-spline in the third direction. 
\begin{theorem}\label{thm:dual-basis-construction}
Given  a dual-compatible spline manifold
  space let
  \begin{equation*}
  \B^*= \left \{  \Lambda _{\bA} ,  \bA  \in    \A \right \}
\end{equation*}
be a set of functionals  $\Lambda _{\bA}:L^2(\Omega) \rightarrow
\mathbb{R}$ such that: 
\begin{itemize}
\item if $\omega_i$ is a structured chart, then for all $\bA\in \A_i$, we have 
\begin{equation}
  \label{eq:tp-Lambda-defn2}
 \Lambda_{\bA}(\phi)=  (\lambda [\Xi_{\bA,i,1}]\otimes \ldots \otimes
 \lambda[\Xi_{\bA,i,d}])(\phi |_{\Omega_i} \circ \pi_i ), \qquad \forall \phi \in L^2(\Omega),
\end{equation}
where the knot vectors $\Xi_{\bA,i,k}$ are given as in \eqref{eq:tp-B-spline-defn};
\item  if $\omega_i$  is an unstructured vertex chart and $B_\bA$, for $\bA\in \A_i$, is the corresponding unstructured vertex function, then let $Q_{i,\ell}=\pi_i(q_{i,\ell})$ and 
\begin{equation}\label{eq:unstructured-vertex-chart-dual-functionals}
  \Lambda_{\bA}(\phi)= \frac{1}{k_i} \sum^{k_i}_{\ell = 1}\Lambda_{\bA,i,\ell}(\phi |_{Q_{i,\ell}}), \quad \forall \phi \in L^2(\Omega),
\end{equation}
where 
\begin{equation*}
  \Lambda_{\bA,i,\ell}(\phi )= (\lambda [\Xi_0]\otimes \ldots \otimes
 \lambda[\Xi_0])(\phi \circ \pi_i \circ (\hat\psi_{i,\ell})^{-1} ),
\end{equation*}
for all  $\phi \in L^2(Q_{i,\ell})$, with $\Xi_0 =  [0,\ldots,0,1]$; and
\item  if $\omega_i$ is an unstructured edge chart and $B_\bA$, for $\bA\in \A_i$, is a corresponding unstructured edge function, then let ${Q}_{i,\ell} = \pi_i(\twoD{q}_{i,\ell} \times \left]a_{i,3},b_{i,3}\right[)$, and
\begin{equation}\label{eq:unstructured-edge-chart-dual-functionals}
  \Lambda_{\bA}(\phi)= \frac{1}{k_i} \sum^{k_i}_{\ell=1}\Lambda_{\bA,i,\ell} (\phi |_{{Q}_{i,\ell}}), \quad \forall \phi \in L^2(\Omega),
\end{equation}
where 
\begin{equation*}
  \Lambda_{\bA,i,\ell}(\phi)= (\lambda [\Xi_{0}] \otimes
 \lambda[\Xi_{0}] \otimes  \lambda[\Xi_{\bA,i,3}])(\phi \circ \pi_i \circ ((\hat\psi_{i,\ell})^{-1},\zeta_3) ),
\end{equation*}
for all $\phi \in L^2({Q}_{i,\ell})$.
\end{itemize}
The set $\B^*$ is dual to $\B$, that is 
\begin{displaymath}
  \Lambda _{\bA} (B_{\bA'}) = \delta_{\bA \bA'}, \quad \forall \bA, \bA' \in \A.
\end{displaymath}
\end{theorem}
\begin{proof}
To prove this theorem, we need to show that \eqref{eq:orthogonal-tilde} is fulfilled. From the dual compatibility condition, it follows that \eqref{eq:orthogonal-tilde} is fulfilled for all $\bA \in \A_i$ and $\bA' \in \A$, where $\omega_i$ is a structured chart. For unstructured charts $\omega_i$, it is obvious that
\begin{equation*}
\Lambda_{\bA_i}(B_{\bA_i'}) = \delta_{\bA_i \bA_i'}, \; \forall \bA_i \in \A_i, \text{ and } \forall \bA_i' \in \A_i \cup (\A \setminus \tilde \A_i).
\end{equation*}
What remains to be shown is that $\Lambda_{\bA_i}(B_{\bA_i'}) = 0$ for all $\bA_i' \in \tilde\A_i \backslash \A_i$. This becomes clear, as for all unstructured functions $B_\bA$ the functional $\Lambda_\bA$ evaluates to zero for all piecewise polynomials that are zero at the extraordinary features. This concludes the proof.
\end{proof}
In the following section we show approximation properties for a simplified configuration. Note that the results may be extended to more general configurations.

\subsection{Approximation properties of spline manifolds on uniform meshes}
\label{sec:properties-dual-compatible}

Having a stable dual basis, one can define 
a projection operator and use it to prove approximation properties of the
B-spline manifold space.

Here we  consider $h$-refinement: a spline manifold
$\Omega$ is given, with an initial mesh
$\T_{h_0}$ over it.  From this we construct a family of meshes $\{ \T_{h}\}
$ by mesh refining each structured chart. The refinement has to be consistent with the 
manifold structure, fulfilling Definition
\ref{def:proto-mesh}, in particular
\eqref{eq:mesh-compatile-with-transitions}, and  Definitions
\ref{def:unstructured-chart-2d}--\ref{def:unstructured-vertex-chart-3d}.
Note that during $h$-refinement the structured charts and  transition functions between
them are unchanged, while the unstructured charts are modified in
order to fulfill the mentioned Definitions.
Accordingly, we assume that a family of nested spaces $\{ \S_{h}\}$ is
given. The subscript $h$ always denotes the dependence on the refinement level.
 
A physical manifold $\Sigma$ is given on the initial mesh through a parametrization $\f G
\in (\S_{h_0})^3$, as in \eqref{eq:G}. For simplicity, we consider a
bivariate $\Omega$, then $\Sigma$ is a closed surface in the space $\R^3$. During $h$-refinement  $\f G$
is kept unchanged. This gives a sequence of nested isogeometric spaces $\V_h$ via \eqref{eq:isogeometric_space}.

To study the approximation properties of $\V_h$ we follow the approach 
in, e.g.,  \cite{bazilevs2006isogeometric,beirao2013analysis,da2012anisotropic} for structured spline spaces,
which  applies to the present framework as well. To keep it simple, we focus on a simple configuration: we
consider a  manifold  obtained
merging tensor-product patches with continuity $C^{p-1}$ but around
the extraordinary vertices, where the continuity is only $C^{0}$. This
configuration is referred to as \emph{multi-patch B-splines with
  enhanced smoothness} as  in \cite{Buchegger2015}, where it has been
first introduced and studied in the context of isogeometric analysis. On
the coarsest mesh $\T_{h_0}$,  the length of the  $C^{0}$ lines is
$p+1$ element edges, which corresponds to the condition that a function
in $\{ \S_{h_0}\}$  is $C^{0}$  across the patch interfaces if and
only if it has  an extraordinary vertex in the
closure of its support. The structured charts for this configuration
can be taken as the union of each pair of patches that have a 
common interface. This is not the only possibility but it simplifies the
next steps. In this case  each transition
 function $\psi_{i,j}$ between structured charts is the composition of a
 translation and a rotation by a multiple of $\frac{\pi}{2}$.
 To further simplify,  we assume the mesh on each structured chart is
 uniform with  mesh-size $h$. 
  
We assume by construction that each space  $\S_{h}$ is
\emph{complete}. By this, we mean that for every
structured chart $\omega_i$ the set $\tilde \B_i$ spans  
all piecewise  polynomials of degree $p$  and continuity stated above.  

The space $L^2(\Omega)$ is defined as 
\begin{equation*}
	L^2(\Omega) = \left\{ \phi : \Omega\rightarrow\R 
	\; | \; \phi\circ\pi_i \in L^2(\omega_i) \; \mbox{ for all structured charts }\omega_i \right\},
\end{equation*}
having the corresponding norm 
\begin{equation}\label{eq:L2-elementwise}
	\| \phi \|^2_{L^2(\Omega)} = \sum_{Q\in\T_h} \| \phi\circ\pi_{i,Q}\|^2_{L^2(\pi_{i,Q}^{-1}(Q))},
\end{equation}
where $\omega_{i,Q}$ is a structured chart covering $Q$, i.e. $Q\subseteq \pi_{i,Q}(\omega_{i,Q})$. 
Due to the isometry of the transition function, the $L^2$-norm of a function defined on a chart fulfills
$	\| \varphi \|^2_{L^2(\omega_{i,j})} = \| \varphi\circ\psi_{j,i}
\|^2_{L^2(\omega_{j,i})}$ for all structured charts $\omega_i$, $\omega_j$. Hence, the $L^2$-norm on $\Omega$ is well-defined. 
Moreover, the definition of the $L^2$-norm is independent of the level of refinement. 
We define bent Sobolev spaces $\mathcal{H}^k(\Omega)$ in the same
fashion. Bent Sobolev
spaces are piecewise Sobolev spaces with some regularity at the
element interfaces, see
\cite{bazilevs2006isogeometric,beirao2014actanumerica}. For example,
the space $\mathcal{H}^{p+1}(\Omega)$ is defined as 
the closure of the space of piecewise $C^{\infty}$ functions having
the same continuity at the mesh lines of the space   $\mathcal{S}$,
with respect to the norm
\begin{equation}\label{eq:bent-sobolev-1}
  	\| \phi \|^2_{\mathcal{H}^{p+1}(\Omega)} = 	\| \phi
        \|^2_{L^2(\Omega)} + \sum_{k=1}^{p+1}  | \phi |^2_{\mathcal{H}^{k}(\Omega)}, 
\end{equation}
where
\begin{equation}\label{eq:bent-sobolev-2}
  | \phi |^2_{\mathcal{H}^{k}(\Omega)}  =    \sum_{Q\in\T_h}
   | \phi\circ\pi_{i,Q}|^2_{H^k(\pi_{i,Q}^{-1}(Q))}
\end{equation}
and  $ | \cdot |_{{H}^{k}(q)} $ is the usual $k$-th order Sobolev seminorm.
Obviously, all the spaces and norms can be defined accordingly on
subdomains of $\Omega$ and are independent of the level of refinement. 

 Using the dual basis defined in  Theorem
 \ref{thm:dual-basis-construction} we introduce a projection operator $\PiOmega \; : \; L^2(\Omega) \rightarrow \mathcal{S}_h$  onto the
 B-spline manifold space via 
 \begin{equation}\label{eq:projector}
   \phi \mapsto \PiOmega(\phi)  =  \sum_{\bA\in\A_h} \Lambda_{\bA}(\phi) B_{\bA}
 \end{equation}
The projector is $L^2$ stable, uniformly with respect to $h$, i.e. 

\begin{equation}
  \label{eq:L2-stab-projector}
  \|  \phi -  \PiOmega(\phi) \|_{L^2(\Omega)}  \leq C\| \phi
  \|^2_{L^2(\Omega)}, \qquad \forall \phi \in L^2(\Omega),
\end{equation}
 which follows directly from the $L^2$ stability of the dual basis as
 defined in Theorem \ref{thm:dual-basis-construction}, see for example \cite{beirao2014actanumerica}.  To prove its approximation properties we use
the following Lemma.
\begin{lemma}
  \label{lemma:bramble-hilbert}
Let ${\mathcal{S}}_{\refbox}$ be the  space of all  piecewise polynomials with
respect to a uniform Cartesian mesh of a box $\refbox \in \mathbb{R}^2$, 
such that 
\begin{itemize}
\item the mesh is formed by up to $2p+1$ elements per direction with
  meshsize $h$;
\item the polynomial degree is $p$ in each direction; and
\item the continuity is  $C^{p-1}$ globally with the exception of a
  mesh line $\f e$ where the continuity is only $C^{0}$, i.e.,  ${\mathcal{S}}_{\refbox} \subset
  C^{p-1}(\refbox \setminus\f  e) \cap  C^{0}(\refbox)$.
\end{itemize}
Let $\mathcal{H}^{p+1}(\refbox) $ be the bent Sobolev space
associated to ${\mathcal{S}}_{\refbox}$. Then 
for all $\phi\in \mathcal{H}^{p+1}({\refbox})  $ there exists a  $\sigma \in
{\mathcal{S}}_{\refbox}$  such that  
\begin{equation}
  \label{eq:bramble-hilbert}
  	\|\phi -\spline\|_{L^2(\refbox)}  \leq C  h^{p+1}  | \phi |_{\mathcal{H}^{p+1}(\refbox)}  
\end{equation}
with a constant $C$ only dependent on $p$.
\end{lemma}
\begin{proof}
The size of $R$ depends on the number of elements $\f n = (n_1,n_2)$ in each direction and the element size $h$. 
Given $\phi \in \mathcal{H}^{p+1}(\refbox) $ there exists  indeed $\spline \in
{\mathcal{S}}_{\refbox}$ such that
\begin{displaymath}
   	\|\phi-\spline\|_{L^2(\refbox)}  \leq C(p,h, \f n, \f e ) | \phi |_{\mathcal{H}^{p+1}(\refbox)}  
\end{displaymath}
with $ C(p,h, \f n,{\bf e})$ independent of $\phi$. The proof is the same as for the
classical Bramble-Hilbert lemma, see, e.g.,
\cite{bazilevs2006isogeometric}. The dependence of the constant  with
respect to $h$, that is $C(p,h, \f n,\f e) = C(p, \f n, \f e )  h^{p+1}$, follows from a
scaling argument. Finally, there are a finite number of different
configurations for $\f n$ and $\f e$, therefore we can set $C(p) = \max_{\f n,\f e} C(p, \f n, \f e ) $.
\end{proof}
See Figure \ref{fig:rectangle-support-extension} for a possible configuration of $\refbox$. Here, 
the line of $C^0$ continuity is shown in blue.
\begin{figure}[!ht]
	\centering
	\includegraphics[width=0.3\textwidth]{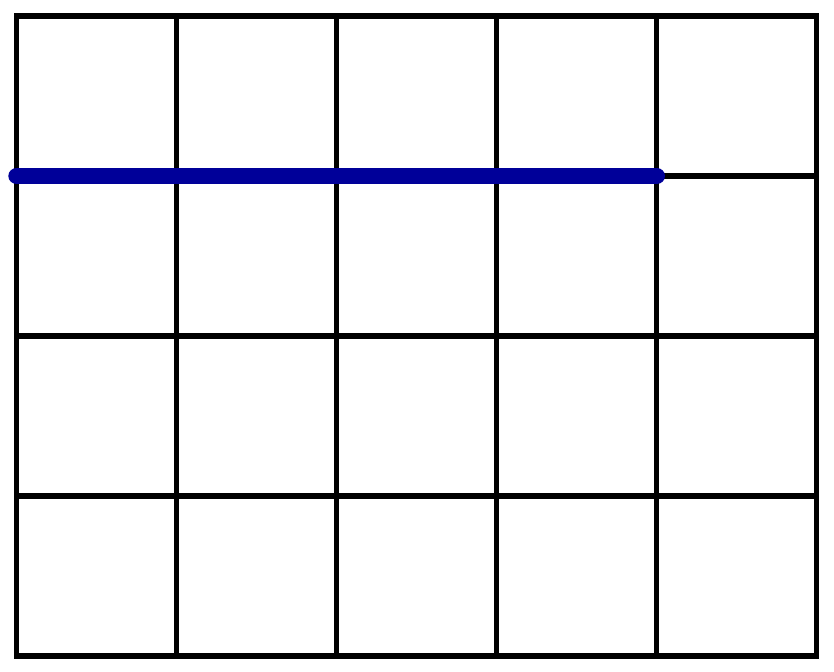}
    	\caption{Possible configuration of a box $\refbox$ and mesh line $\f e$ for degree $p=2$}\label{fig:rectangle-support-extension}
\end{figure}

Given  $Q\in\T_h$, we define  $\refext{Q} \subset \Omega$ in the
following way: for each structured chart $\omega_i$, $\refext{Q} \cap \Omega_i$ is the minimal box
containing all the supports of functions in $\B_h$ whose support
includes $Q$. 
We can now state the local approximation estimate.
\begin{theorem}\label{thm:local-approximation-parameter}
Under the assumptions of this Section, for $\phi \in
\mathcal{H}^{p+1}(\refext{Q})  $ it holds
\begin{equation}\label{eq:local-error-bound}
	\| \phi - \PiOmega(\phi) \|_{L^2({Q})} \leq C  h^{p+1}  | \phi |_{\mathcal{H}^{p+1}(\refext{Q})}  ,
\end{equation}
where the constant $C$ depends only on $p$.
\end{theorem}
\begin{proof}
The first step of the proof is to show that, given $ \phi\in 
\mathcal{H}^{p+1}(\refext{Q})  $, there is an $s \in \mathcal{S}_h$ such
that 
\begin{equation}\label{eq:local-error-bound-s}
	\| \phi - s \|_{L^2(\refext{Q})} \leq C  h^{p+1}  | \phi |_{\mathcal{H}^{p+1}(\refext{Q})}. 
\end{equation}
  We are in one of two cases, either
  \begin{enumerate}
  \item[(a)]  $\refext{Q} $ contains an extraordinary vertex,  or
  \item[(b)] $\refext{Q} \subset \pi_i(\omega_i ) $ where $  \omega_i$ is a structured chart.
  \end{enumerate}
\begin{figure}[!ht]
\centering
\subfigure[]{\includegraphics[height=0.15\textheight]{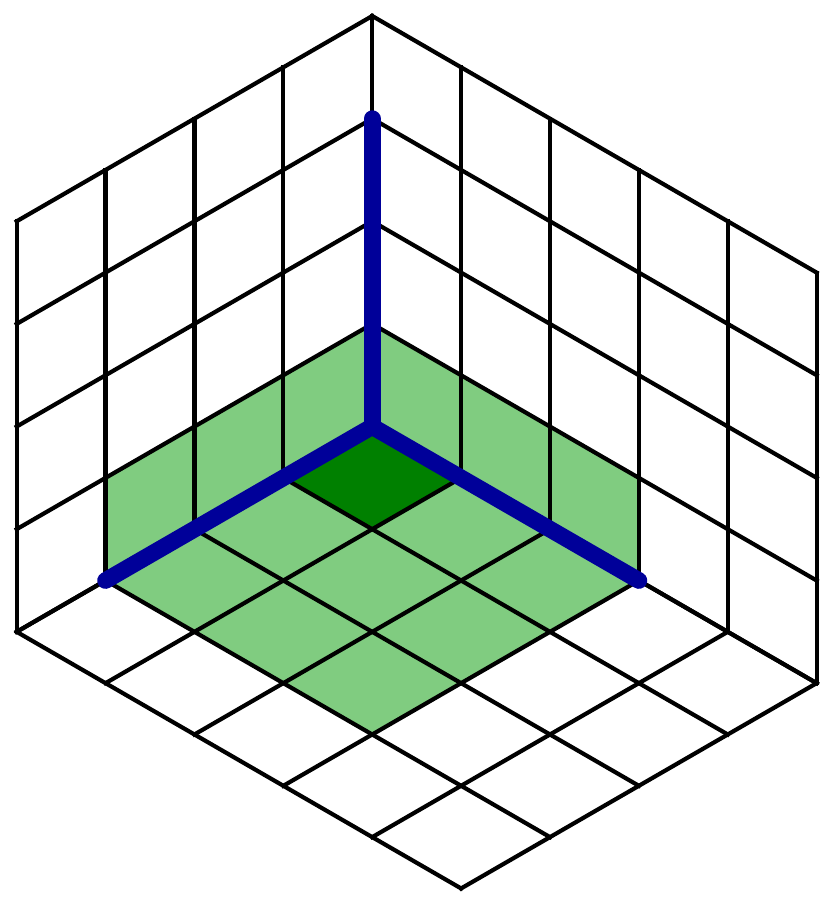}\label{fig:support-extension-a}}
\hspace{0.1\textwidth}
\subfigure[]{\includegraphics[height=0.15\textheight]{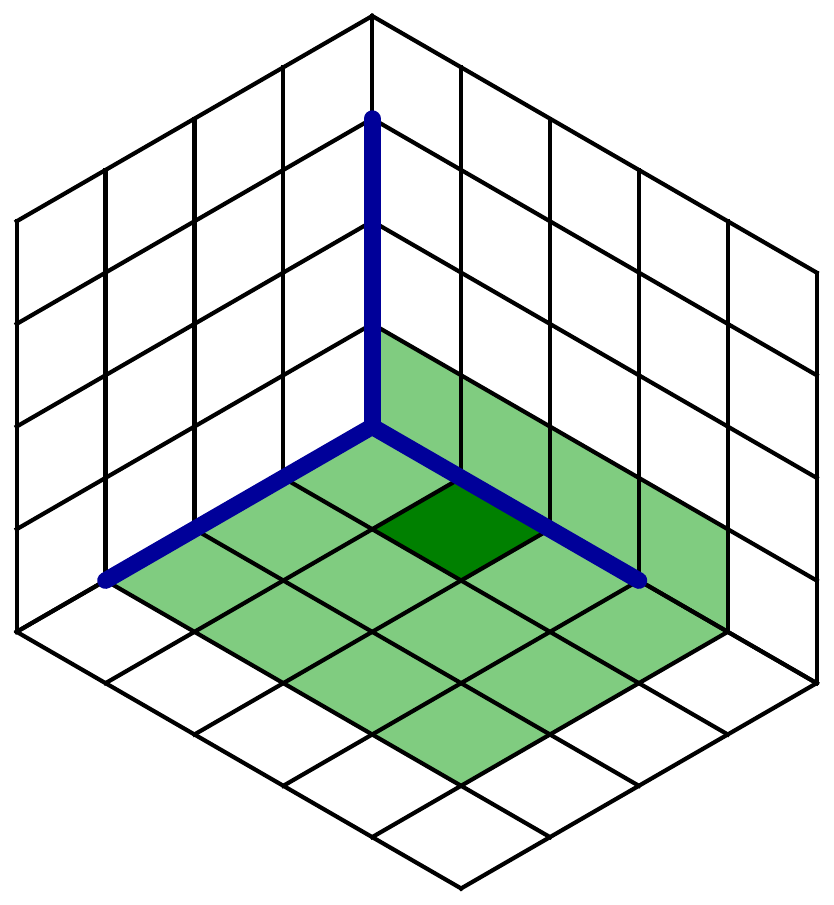}\label{fig:support-extension-b}}
    \caption{Different types of support extensions $\refext{Q}$ for $p=2$}\label{fig:support-extension}
\end{figure}
See Figure \ref{fig:support-extension} for a representation of the two possible cases. In the figure 
the $C^0$-continuity lines are depicted in blue. The dark green element represents $Q$ 
and the light green region represents its support extension $\refext{Q}$. 
Due to the assumption on the length of the $C^0$ lines, case 1) occurs when  $Q$ is adjacent to an extraordinary vertex, i.e., there
    exists an unstructured chart $\omega_i $  such  that $Q\subset
    \pi_i(\omega_i ) $. In this  case we can split   $\refext{Q} $  into
    $\refext{Q}_1, \ldots,  \refext{Q}_{k_i}$  such that each
    $\refext{Q}_\ell$ intersects only  the segment
    $\pi_i( s_{i,\ell})$ (see Definition
    \ref{def:unstructured-chart-2d}). Each 
    $\mathcal{S}_h\rst{\refext{Q}_\ell} $  is a standard tensor-product
    spline space therefore we can use \cite[Section 2.2.2 and Section
    4.4]{beirao2014actanumerica} in order to construct  splines $\spline
    _\ell\in \mathcal{S}_h\rst{\refext{Q}_\ell} $ which approximate $\phi
    \rst{\refext{Q}_\ell }$ and match continuously in the whole
    $\refext{Q} $. Since on the coarsest mesh the $C^0$ lines at the extraordinary 
    vertex cover $p+1$ element edges (see Figure \ref{fig:support-extension}), each  interface between two
    adjacent sets  of $\refext{Q}_1, \ldots,  \refext{Q}_{k_i}$
    is covered by the  $C^0$ continuity lines of   $s \in \mathcal{S}_h$  and
    \eqref{eq:local-error-bound-s} follows. 
In Case 2), we can use  Lemma \ref{lemma:bramble-hilbert}, then  \eqref{eq:local-error-bound-s} follows by
\eqref{eq:bramble-hilbert}.

Having \eqref{eq:local-error-bound-s} and recalling the $L^2$
stability of the projector  $\PiOmega$,  \eqref{eq:local-error-bound} is
derived in the usual way, i.e.
\begin{align*}
  	\| \phi - \PiOmega(\phi) \|_{{L}^{2}({Q})} &=  	\| \phi -\spline -    \PiOmega(\phi-\spline) \|_{{L}^{2}({Q})}
  \\ &\leq C 	\| \phi -\spline \|_{{L}^{2}({\refext{Q}})} 
  \\ & \leq  C h^{p+1}  | \phi|_{\mathcal{H}^{p+1}(\refext{Q})} .
\end{align*}
The stability constant as well as the approximation constant only depend on $p$, which concludes the proof.
\end{proof}
By means of inverse estimates and generalising
\eqref{eq:local-error-bound-s},  \eqref{eq:local-error-bound} can be extended to
higher order Sobolev norms. For  $0 \leq q \leq p+1$  it holds  
\begin{equation}\label{eq:local-error-bound-higher-order}
	| \phi - \PiOmega(\phi) |_{{H}^{q}({Q})} \leq C  h^{p-q+1}  | \phi|_{\mathcal{H}^{p+1}(\refext{Q})}  ,
\end{equation}
where the constant $C$ depends only on $p$ and $q$. The details are
not reported for the  sake of brevity.

We can extend the result to a mapped domain $\Sigma$, in this case a
closed surface in $\R^3$, parametrized by $\f G
\in (\S_{h_0})^3$.  We assume that  $\f G$ is regular, that is,  there exist constants $\underline c, \overline c$, with
$$
\overline c \geq \det (\nabla \f G ^T(\f x)  \nabla \f G (\f x)) \geq \underline c > 0
$$
for all $\f x \in Q$ and for all $Q \in \T_{h_0}$. Note that, in
general,  we
cannot define Sobolev spaces of any order on $\Sigma$, due to the lack of
smoothness of the manifold itself. However  the $L^2$
space on $\Sigma$ can be defined as 
\begin{displaymath}
  L^2(\Sigma) = \{ f: f\circ \f G  \in L^2(\Omega)\}
\end{displaymath}
and the corresponding norm is given via 
\begin{displaymath}
  \| f \|_{L^2(\Sigma)}  = \| f\circ \f G \, (\det (\nabla \f
  G ^T  \nabla \f G ))^{1/4} \|_{L^2(\Omega)}.
\end{displaymath}
The bent Sobolev spaces can be defined similarly,  as in
\eqref{eq:bent-sobolev-1}--\eqref{eq:bent-sobolev-2}. See
\cite{Hebey1996,Hebey2000} for more details about Sobolev spaces on manifolds.
Then, for all  $f \in L^2(\Sigma)$, we can define the isogeometric projector
\begin{displaymath}
  \PiSigma(f) = \PiOmega(f\circ \f G ) \circ \f G ^{-1}.
\end{displaymath}
Approximation properties of $ \PiSigma$
easily follow from the ones of $\PiOmega$ stated in Theorem
\ref{thm:local-approximation-parameter}, following the same approach
of \cite{bazilevs2006isogeometric,dede2015isogeometric}.
\begin{theorem}\label{thm:approximation-sigma}
Under the assumptions of this Section,  for all $f \in \mathcal{H}^{p+1}(\Sigma)$
\begin{equation}\label{eq:error-bound-on-Sigma}
	\| f - \PiSigma(f) \|_{L^2(\Sigma)} \leq C h^{p+1} \| f \|_{\mathcal{H}^{p+1}(\Sigma)}.
\end{equation}
where the constant $C$
depends only on  $\f G$ and  $p$.
\end{theorem}
Note that all the results presented here extend naturally to volumetric domains. In that case, the lines of $C^0$ continuity extend to faces of $C^0$ continuity in a vicinity of the extraordinary vertices and edges.

\section{Conclusion and possible extensions}
\label{sec:conclusion}

We have introduced a new general mathematical framework, based on manifolds, for the definition and the analysis of unstructured spline spaces. As it is done in \cite{grimm1995modeling} the main idea is to decompose the domain into charts, which are meshed with quadrilaterals or hexahedra. Then spline basis functions and dual functionals can be defined locally on each chart. Unstructured charts are necessary to cover extraordinary vertices and edges of the domain.

We have used this framework to generalize the dual-compatibility condition of \cite{beirao2012analysis,beirao2013analysis} to unstructured spline spaces, and in particular to analyze the approximation properties of splines with high continuity everywhere except in the vicinity of extraordinary vertices and edges, where the continuity is only $C^0$. Although the analysis was restricted to the low continuity case, the framework allows for the definition of spline functions with higher smoothness, and their analysis will be the aim of future work.

In our definitions the physical domain is necessarily a manifold. However, since we are defining the charts in the parametric domain, and not in the physical domain, it is possible to extend our framework to non-manifold domains using special bifurcation charts (such as T-shaped or X-shaped charts for curves, etc.) This could be of interest for certain beam or shell formulations, or for the proper representation of the medial axis or medial surface of an object, for instance.

Finally, for the sake of simplicity we have restricted ourselves
  to B-splines and T-splines on quadrilateral/hexahedral meshes. The
  framework can be easily generalized to other spline spaces, such as
  NURBS or  trigonometric splines.

\section*{Acknowledgements}
The authors were partially supported by the European Research Council
through the FP7 Ideas Starting Grant \emph{HIGEOM}, and by the Italian
MIUR through the PRIN  ``Metodologie innovative nella modellistica differenziale numerica''. This support is gratefully acknowledged.

\begin{appendix}
\section{Spline manifolds with boundary}\label{appendix:boundary}

We can extend the definition of a spline manifold to a manifold with boundary. To do so, we first need to extend Definition~\ref{defi:proto-manifold}, defining a suitable proto-manifold that takes into account the boundary.
\begin{definition}[Manifolds with boundary]
A \emph{proto-manifold with boundary} is a generalization of a
proto-manifold, where the charts $\{\omega_i\}_{i=1,\ldots,N}$  are
given as $\omega_i = \inter{\omega_i} \cup \gamma_i$, such
that $\inter{\omega_i}$ are open polytopes forming a
standard proto-manifold (with transition domains
$\inter{\omega_{i,j}}$ and transition functions
${\psi}_{i,j}$) and each $\gamma_i
\subset \partial\omega_i$ is a part of the boundary of the
chart $\inter{\omega_i}$. Moreover, the transition domains fulfill $\omega_{i,j} =
\inter{\omega_{i,j}}\cup\gamma_{i,j}$, with $\gamma_{i,j}
= \gamma_i \cap \partial\omega_{i,j}$, and the transition functions
are the continuous extensions of
${\psi}_{i,j}$ and map $\gamma_{i,j}$ onto
$\gamma_{j,i}$ and $\inter{\omega_{i,j}}$ onto
$\inter{\omega_{j,i}}$. 

Similar to the standard parameter manifold in Definition~\ref{def:parameter-manifold}, we can define the \emph{parameter manifold with boundary} $\Omega$ via the equivalence relation induced by the transition functions. Since the transition functions always map the interior onto the interior and the boundary onto the boundary, the parameter manifold $\Omega$ can be separated into $\inter{\Omega}$ and the boundary denoted by $\Gamma$.
\end{definition}

This definition of the boundary of the (open) parameter manifold is
equivalent to the classical definition of a manifold with boundary, as
discussed in \cite{grimm1995modeling}. In this case every boundary point of the manifold has a neighborhood that is homeomorphic to the \emph{half $d$-ball}. Note that the boundary $\Gamma$ itself can be interpreted as a topological manifold of dimension $d-1$.

For the definition of the spline spaces, we assume the following.
\begin{assumption}
The local boundary $\gamma_i$ is conforming with respect to the elements, i.e. there exists a subset of faces of elements $q \in \tau_i$ that forms a mesh for the boundary $\gamma_i$.
\end{assumption}
To be able to define manifolds with boundary containing non-convex
features, we need additional types of charts, so called boundary
charts. To avoid the tedious formal definition of boundary charts, we
present figures that should explain the ideas behind. In Figure
\ref{fig:boundvert2D} we present different types of two-dimensional
boundary vertices, regular boundary vertices (a) - (c), as well as a
non-regular boudary vertex (d). The vertices in
Figures~\ref{fig:boundvert2D-a} and \ref{fig:boundvert2D-b} can be
covered by the boundaries of structured charts (see Figure
\ref{fig:boundary-charts-2d-a} -- \ref{fig:boundary-charts-2d-c}). For
the vertex in Figure~\ref{fig:boundvert2D-c} we need a special
boundary chart (see Figure \ref{fig:boundary-charts-2d-d}), associated
with the function which is non-zero at the corner. Note that the boundary 
vertex in Figure~\ref{fig:boundvert2D-d} is discarded since it does not 
allow for functions that are non-zero at the corner.
\begin{figure}[!ht]
    \centering
\subfigure[]{\includegraphics[width=0.15\textwidth]{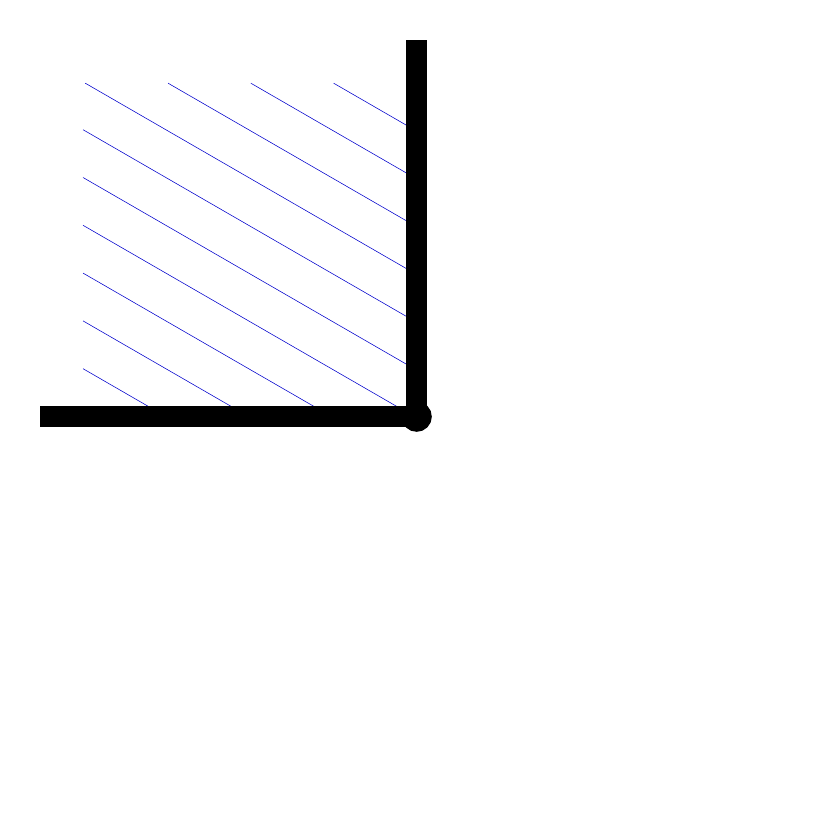}\label{fig:boundvert2D-a}}\qquad
\subfigure[]{\includegraphics[width=0.15\textwidth]{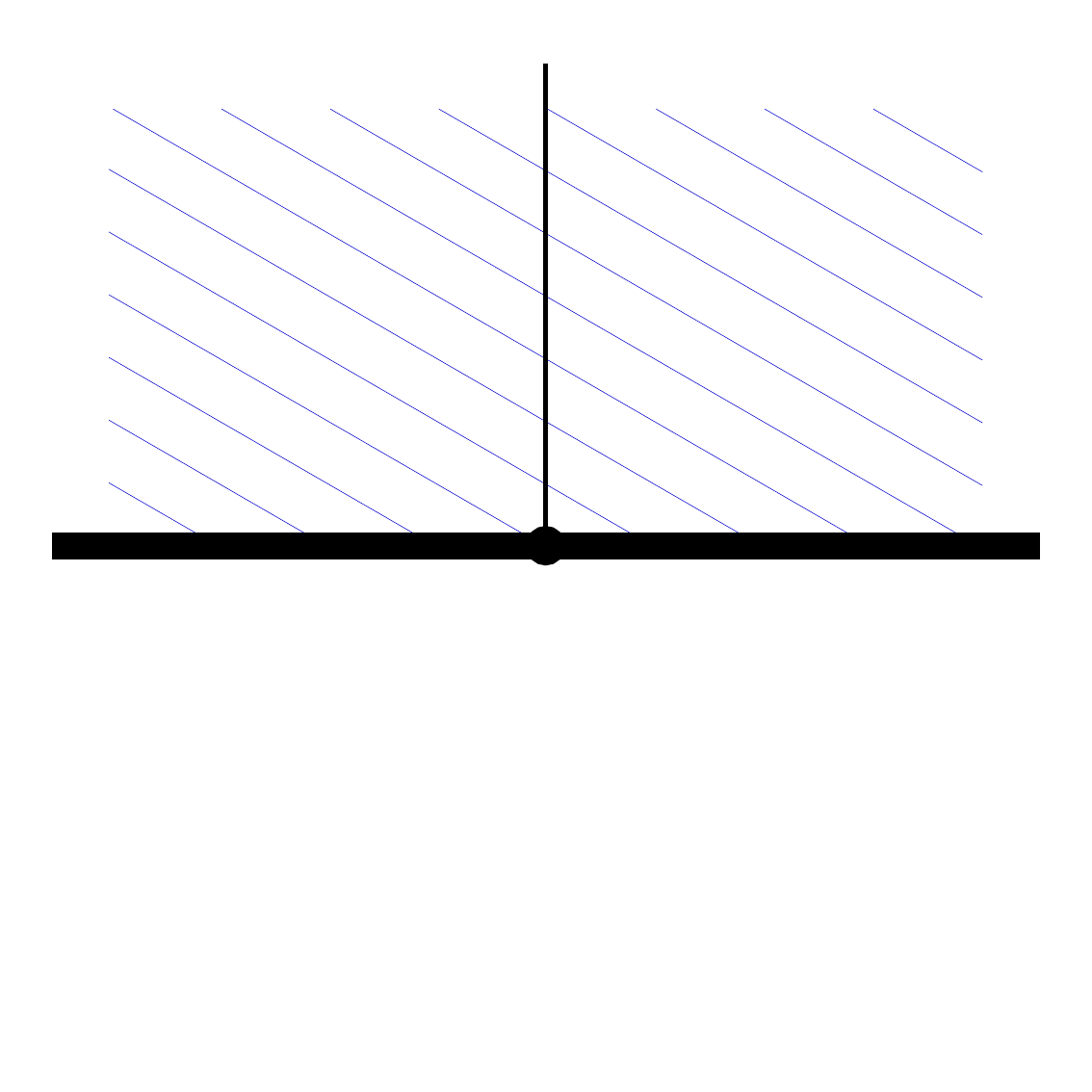}\label{fig:boundvert2D-b}}\qquad
\subfigure[]{\includegraphics[width=0.15\textwidth]{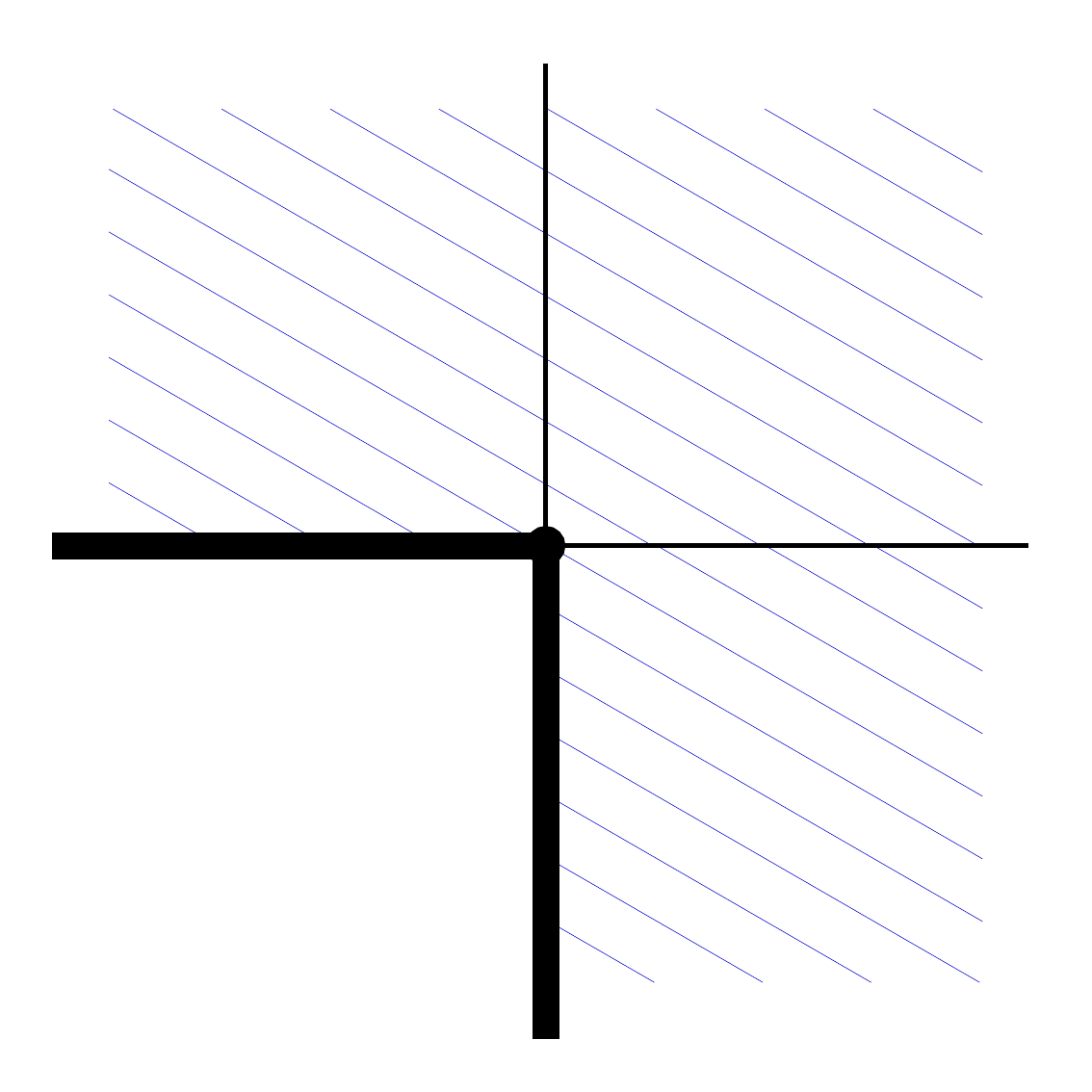}\label{fig:boundvert2D-c}}\qquad
\subfigure[]{\includegraphics[width=0.15\textwidth]{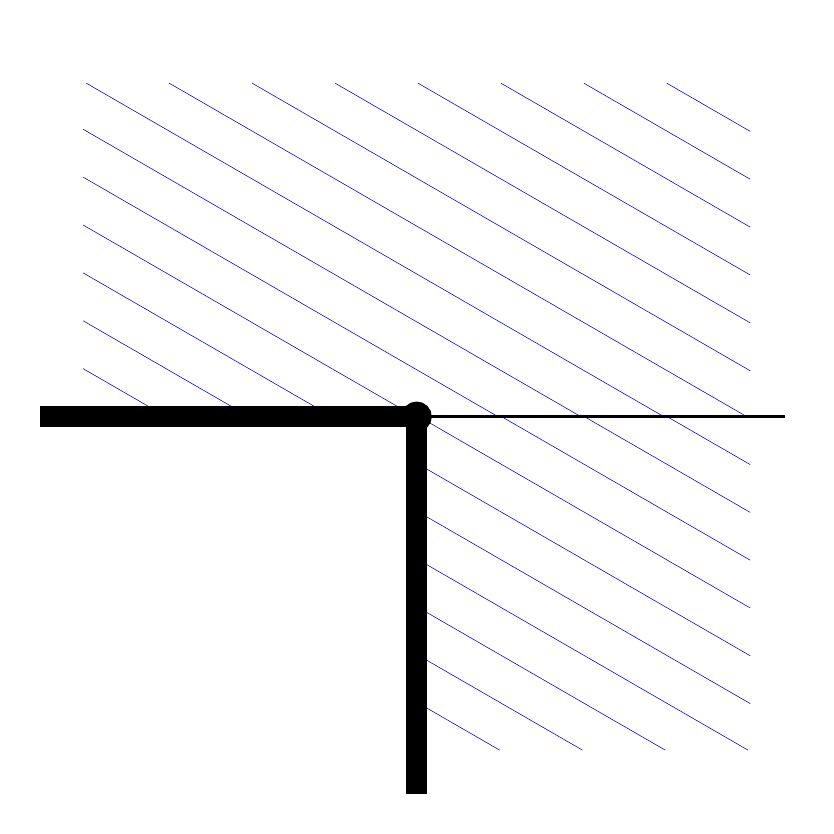}\label{fig:boundvert2D-d}}
    \caption{Different types of valid (a)--(c) and non-valid (d) boundary vertices in 2D}\label{fig:boundvert2D}
\end{figure}

\begin{figure}[!ht]
    \centering
\subfigure[]{\includegraphics[width=0.15\textwidth]{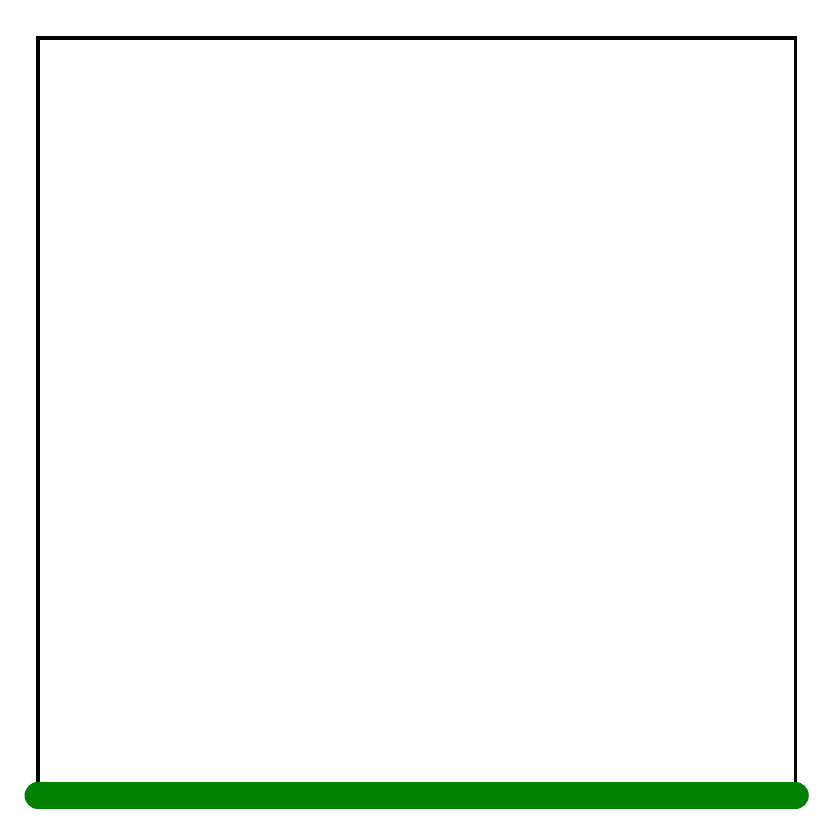}\label{fig:boundary-charts-2d-a}}\qquad
\subfigure[]{\includegraphics[width=0.15\textwidth]{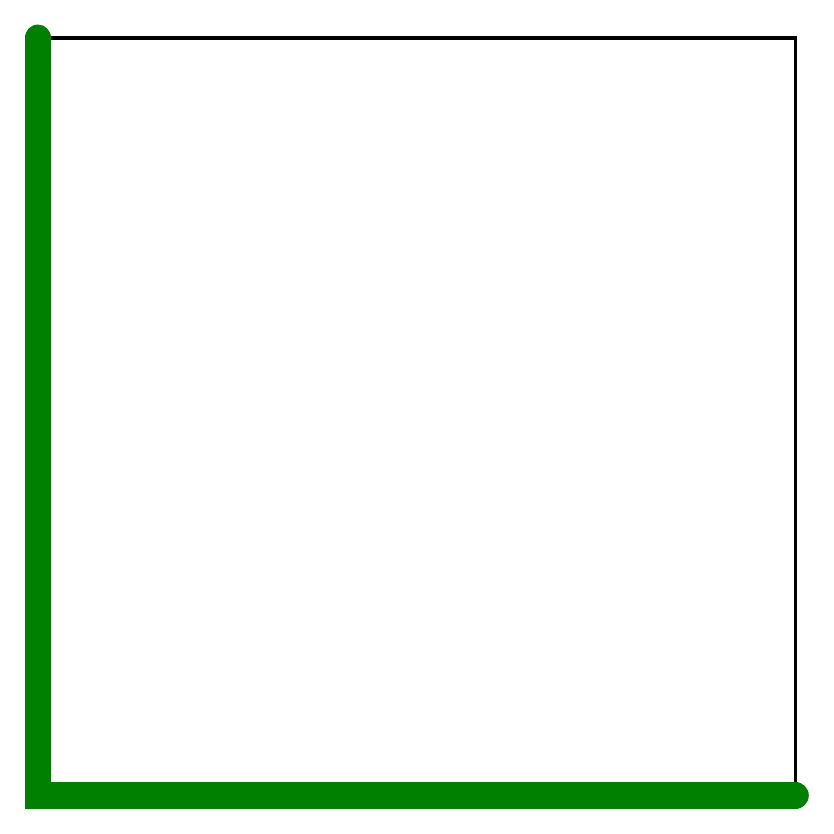}\label{fig:boundary-charts-2d-b}}\qquad
\subfigure[]{\includegraphics[width=0.15\textwidth]{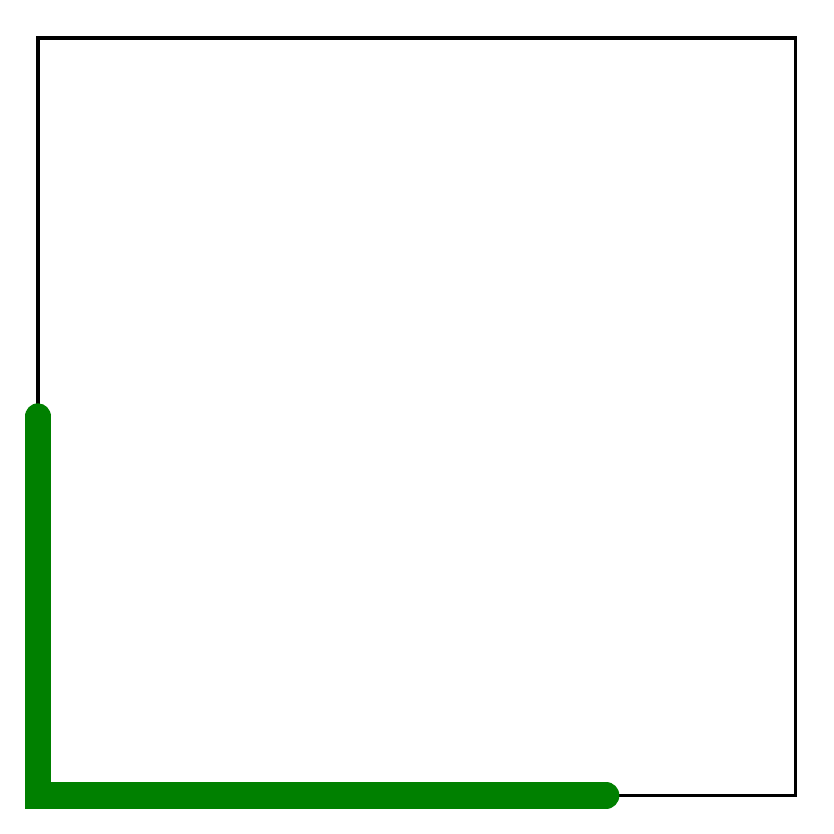}\label{fig:boundary-charts-2d-c}}\qquad
\subfigure[]{\includegraphics[width=0.15\textwidth]{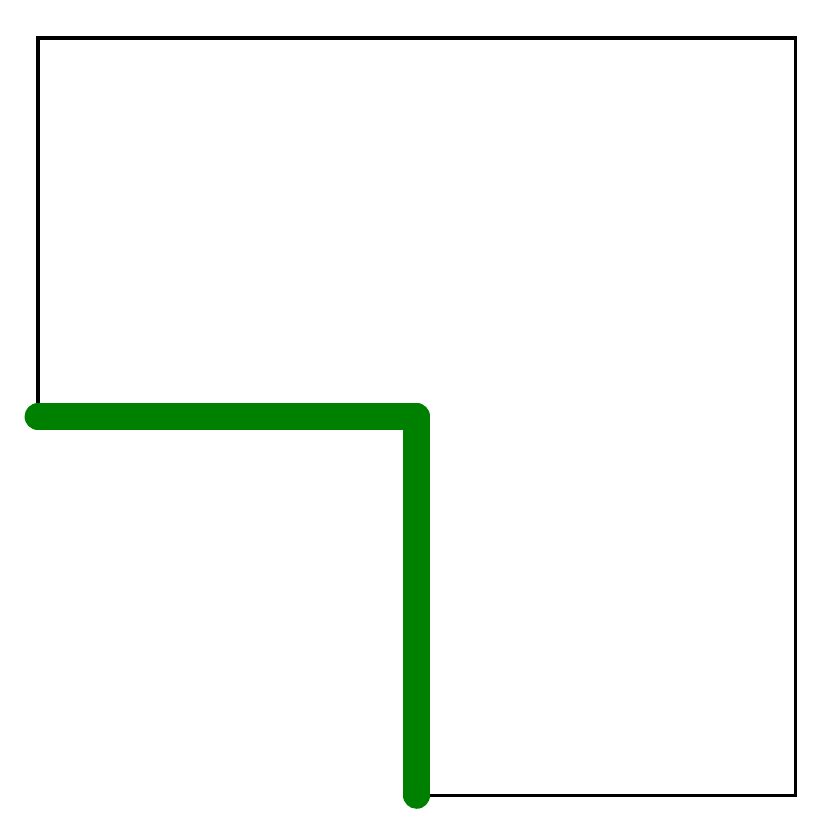}\label{fig:boundary-charts-2d-d}}
    \caption{Structured charts with boundary (a)--(c) and boundary chart (d) in 2D}\label{fig:boundary-charts-2d}
\end{figure}

In the three-dimensional case, there are more different configurations to consider, which are listed in Table \ref{tab:boundary-vertices}. 
Here, we need to consider two different types of boundary charts.
\begin{table}[ht]
\centering
\begin{tabular}{ll}
structured edge & regular boundary edge \\
structured vertex & regular boundary vertex \\
 & hanging boundary vertex \\
unstructured vertex & partially unstructured boundary vertex
\end{tabular}
    \caption{Classification of valid boundary edges and vertices in 3D}\label{tab:boundary-vertices}
\end{table}

Figures \ref{fig:charts-w-boundary-3d} and \ref{fig:boundary-charts-3d} depict several possible three-dimensional charts with boundary. In Figures \ref{fig:charts-w-boundary-3d-a}--\ref{fig:charts-w-boundary-3d-c} we show several examples of structured charts with boundary. Figure \ref{fig:charts-w-boundary-3d-d} shows an unstructured edge chart with boundary. 
\begin{figure}[!ht]
    \centering
\subfigure[]{\includegraphics[width=0.15\textwidth]{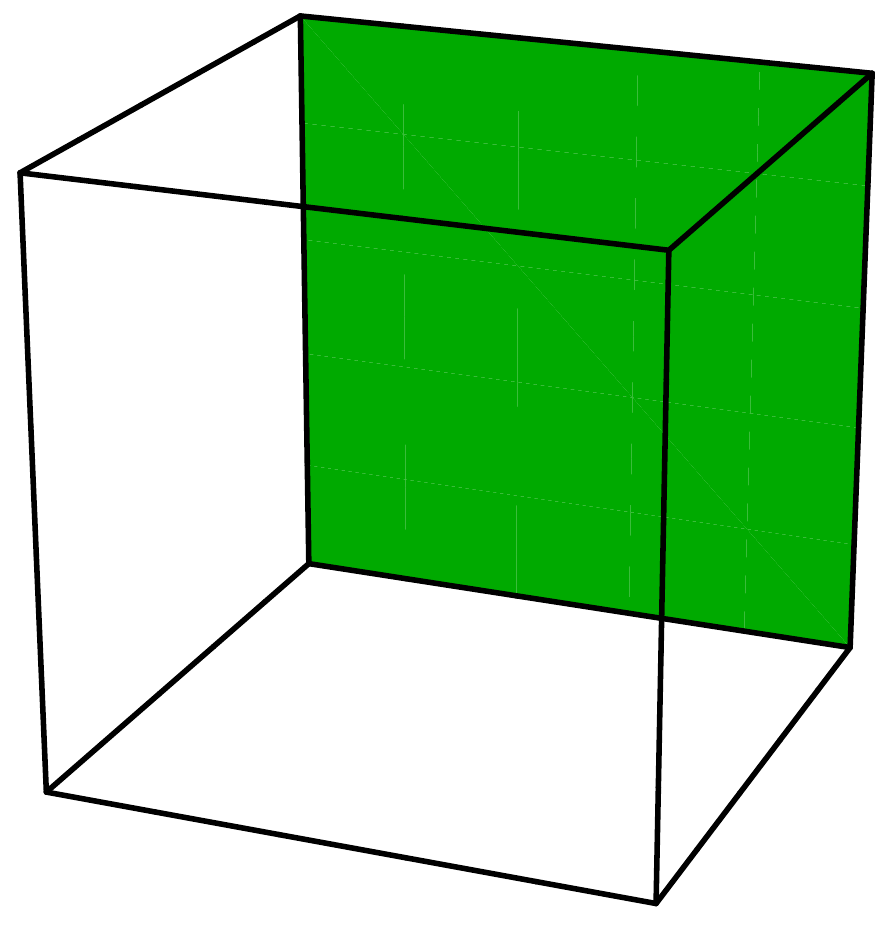}\label{fig:charts-w-boundary-3d-a}}\qquad
\subfigure[]{\includegraphics[width=0.15\textwidth]{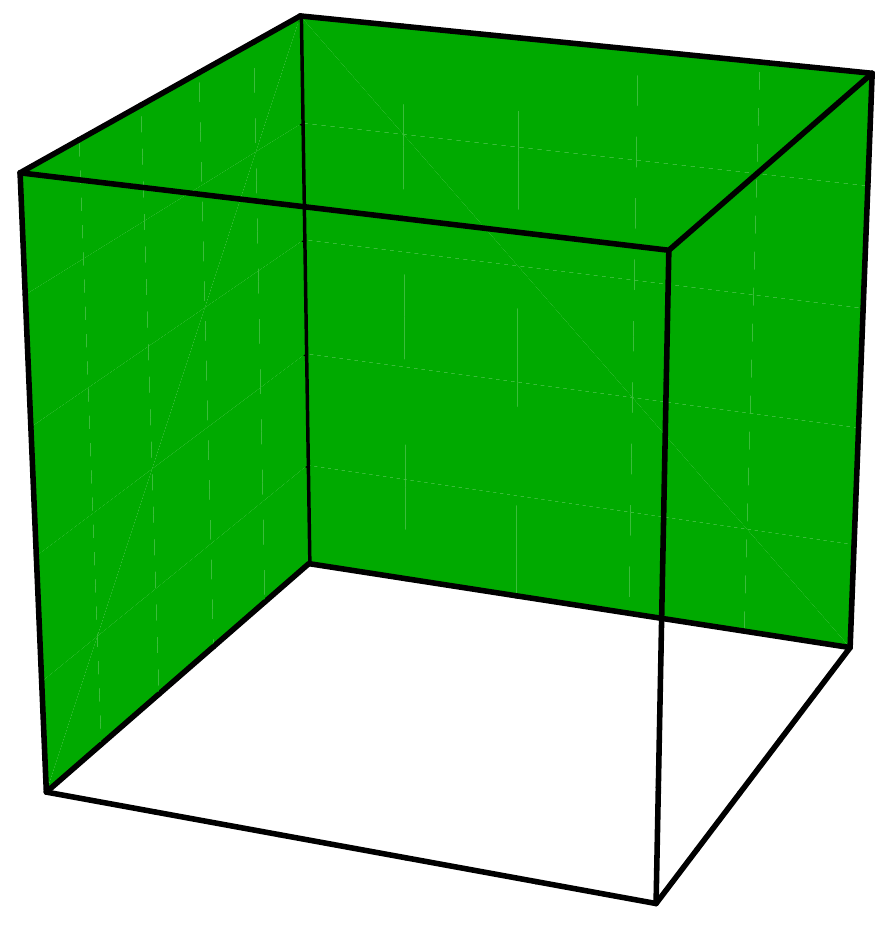}\label{fig:charts-w-boundary-3d-b}}\qquad
\subfigure[]{\includegraphics[width=0.15\textwidth]{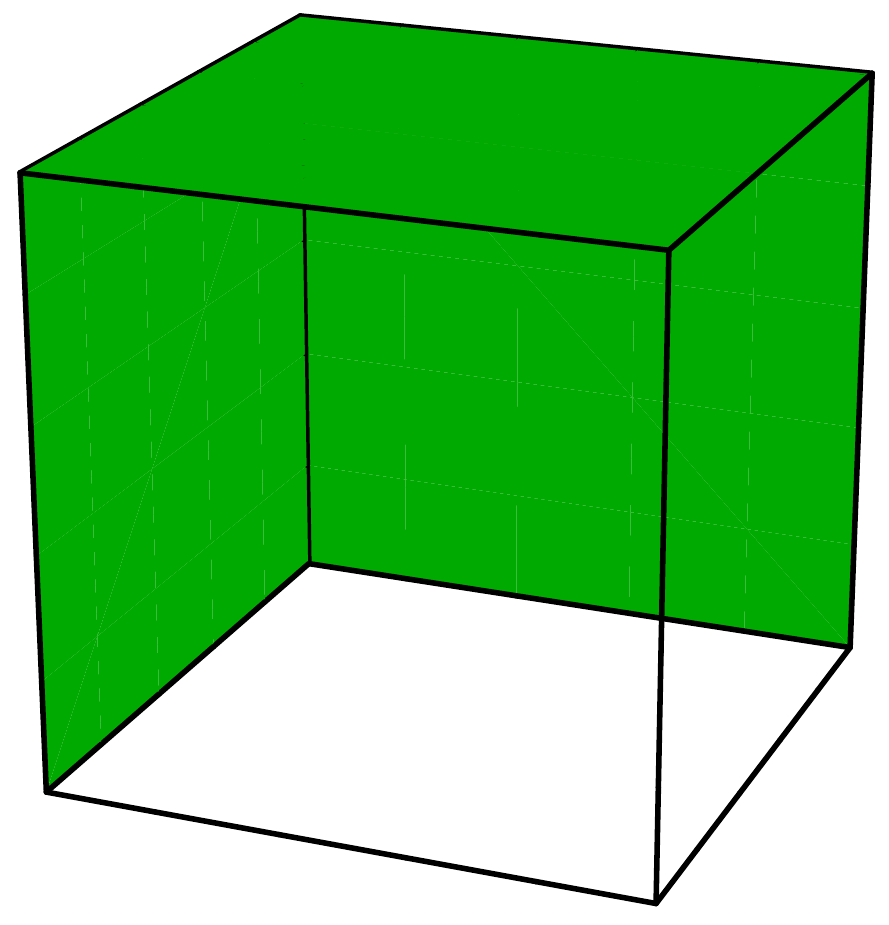}\label{fig:charts-w-boundary-3d-c}}\qquad
\subfigure[]{\includegraphics[width=0.15\textwidth]{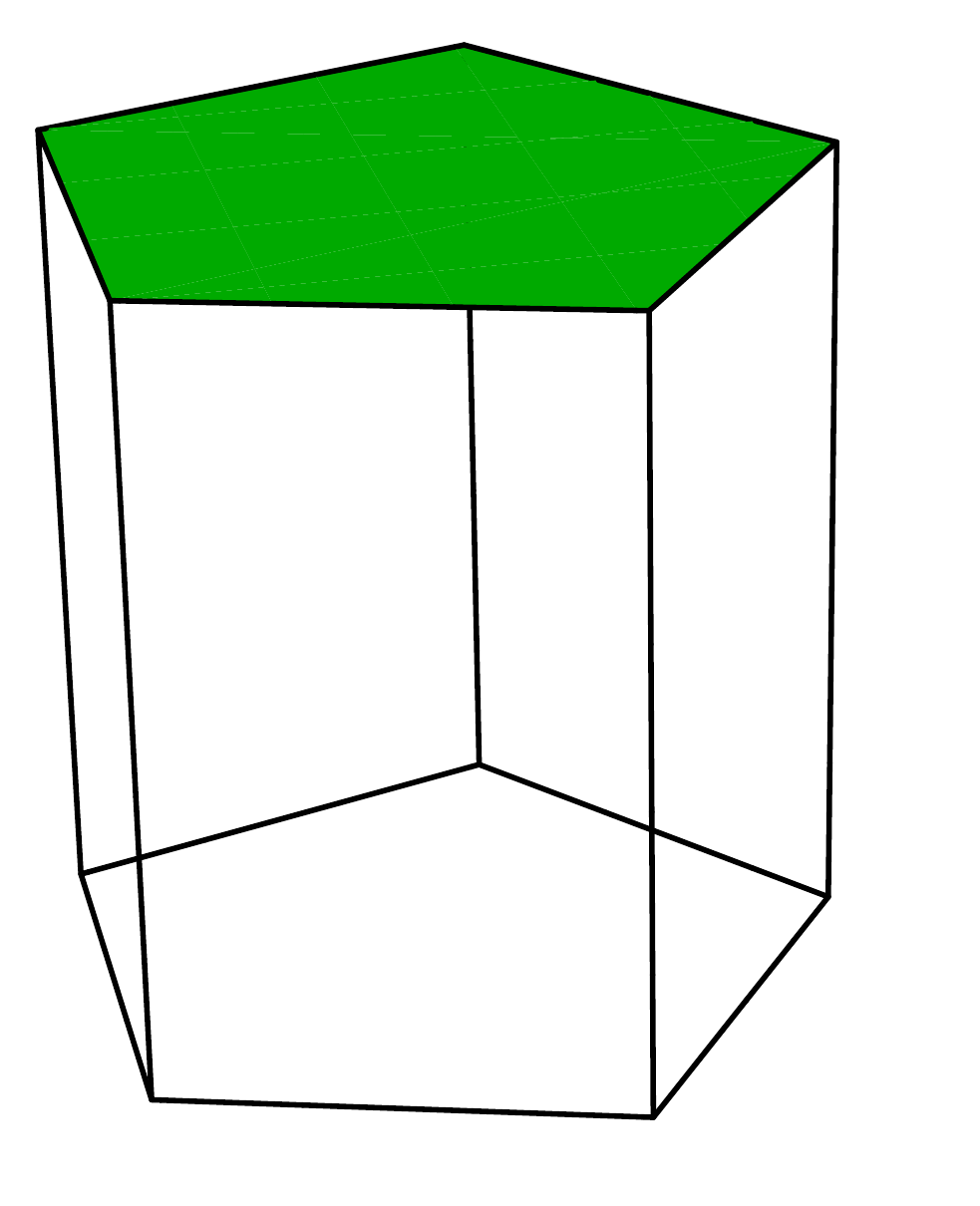}\label{fig:charts-w-boundary-3d-d}}
    \caption{Structured charts with boundary (a)--(c) and unstructured edge chart with boundary (d) in 3D}\label{fig:charts-w-boundary-3d}
\end{figure}
In Figure \ref{fig:boundary-charts-3d} we show the two different types of boundary charts, that are needed to represent all meshes of practical interest. The one in Figure \ref{fig:boundary-charts-3d-a} is a tensor-product of a two-dimensional boundary chart with an interval in the third direction. The chart depicted in Figure \ref{fig:boundary-charts-3d-b} is a boundary chart corresponding to a non-convex vertex at the boundary. One can include more complex boundary configurations by introducing unstructured boundary charts, which we will not consider for the sake of simplicity.
\begin{figure}[!ht]
    \centering
\subfigure[]{\includegraphics[width=0.15\textwidth]{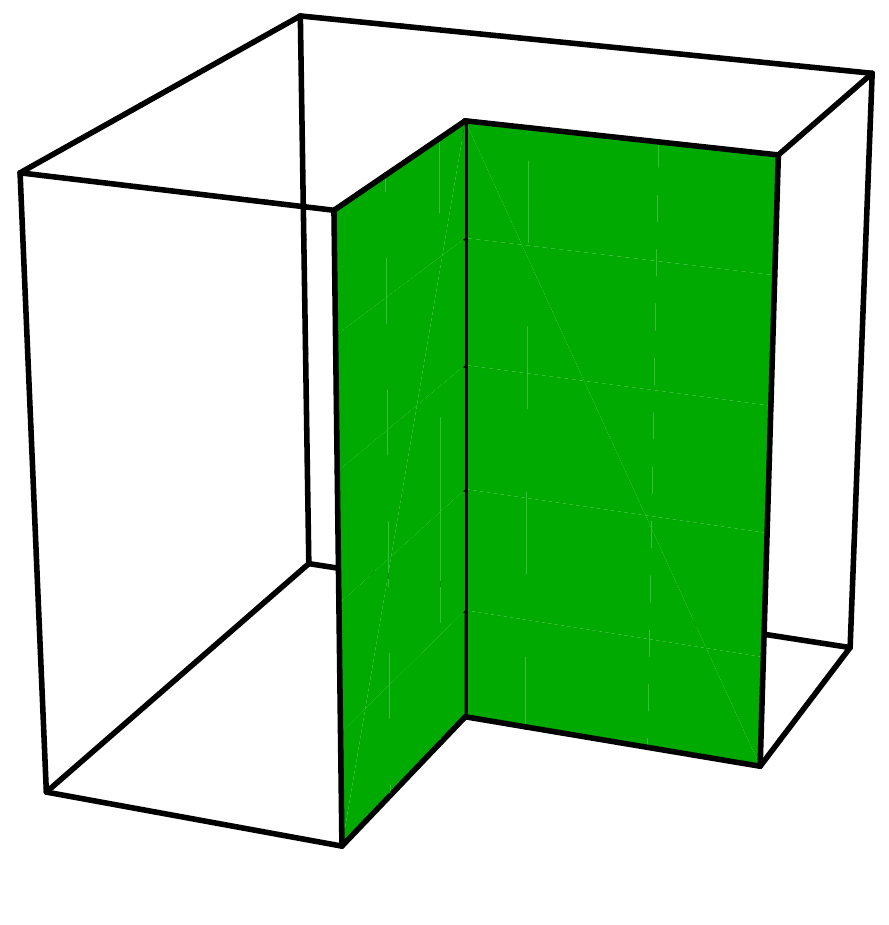}\label{fig:boundary-charts-3d-a}}\qquad
\subfigure[]{\includegraphics[width=0.15\textwidth]{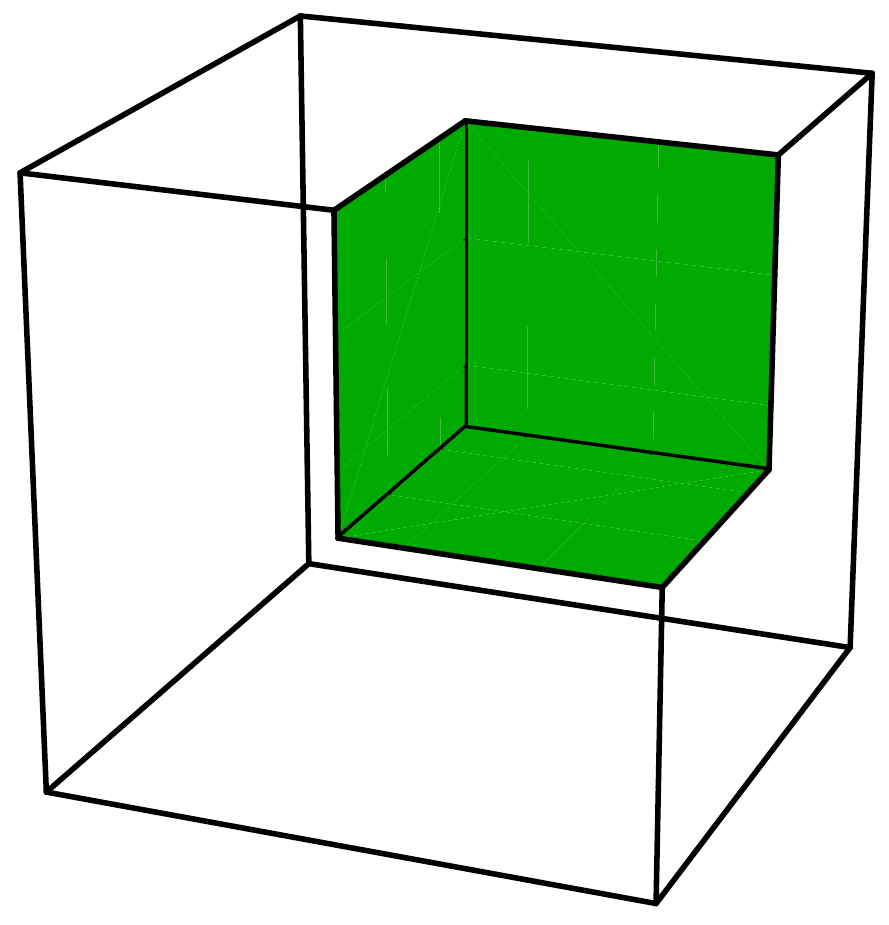}\label{fig:boundary-charts-3d-b}}
    \caption{Boundary charts in 3D}\label{fig:boundary-charts-3d}
\end{figure}

Concerning spline manifold spaces on parameter manifolds with boundary, we need to adjust the condition on proto-basis functions in equation \eqref{eq:support-of-proto-basis-functions} to obtain
\begin{equation}
  \label{eq:support-of-proto-basis-functions-boundary}
  \lim_{\bzeta \rightarrow \partial    \omega_i \backslash \gamma_i} b _{\bA_i}(\bzeta) = 0. 
\end{equation}
With this modification, the spline manifold space can interpolate at the boundary. The theory concerning dual-compatibility and approximation properties presented in Section \ref{sec:analysis-suitable} can be generalized directly to spline manifolds with boundary.
Most importantly, both Theorems \ref{thm:local-approximation-parameter} and \ref{thm:approximation-sigma} extend directly to manifolds with boundary, using the boundary charts introduced here. Thus extending the approximation error bounds to manifolds with boundary of general topology.

\end{appendix}


\begin{thebibliography}{10}
\expandafter\ifx\csname url\endcsname\relax
  \def\url#1{\texttt{#1}}\fi
\expandafter\ifx\csname urlprefix\endcsname\relax\def\urlprefix{URL }\fi
\expandafter\ifx\csname href\endcsname\relax
  \def\href#1#2{#2} \def\path#1{#1}\fi

\bibitem{grimm1995modeling}
C.~M. Grimm, J.~F. Hughes, Modeling surfaces of arbitrary topology using
  manifolds, in: Proceedings of the 22nd annual conference on Computer graphics
  and interactive techniques, ACM, 1995, pp. 359--368.

\bibitem{beirao2013analysis}
L.~Beir{\~a}o~da Veiga, A.~Buffa, G.~Sangalli, R.~V{\'a}zquez,
  Analysis-suitable {T}-splines of arbitrary degree: definition, linear
  independence and approximation properties, Mathematical Models and Methods in
  Applied Sciences 23~(11) (2013) 1979--2003.

\bibitem{Hughes2005}
T.~J.~R. Hughes, J.~A. Cottrell, Y.~Bazilevs, Isogeometric analysis: {CAD},
  finite elements, {NURBS}, exact geometry and mesh refinement, Computer
  Methods in Applied Mechanics and Engineering 194~(39-41) (2005) 4135--4195.

\bibitem{igaBook}
J.~A. Cottrell, T.~J.~R. Hughes, Y.~Bazilevs, Isogeometric {A}nalysis: toward
  integration of {CAD} and {FEA}, John Wiley \& Sons, 2009.

\bibitem{kiendl2010}
J.~Kiendl, Y.~Bazilevs, M.-C. Hsu, R.~W{\"u}chner, K.-U. Bletzinger, The
  bending strip method for isogeometric analysis of {K}irchhoff-{L}ove shell
  structures comprised of multiple patches, Computer Methods in Applied
  Mechanics and Engineering 199~(35) (2010) 2403--2416.

\bibitem{Beirao2011}
L.~B. da~Veiga, A.~Buffa, D.~Cho, G.~Sangalli, Isogeometric analysis using
  {T}-splines on two-patch geometries, Computer Methods in Applied Mechanics
  and Engineering 200~(21-22) (2011) 1787 -- 1803.

\bibitem{Kleiss2012}
S.~K. Kleiss, C.~Pechstein, B.~J{\"u}ttler, S.~Tomar, {IETI} - isogeometric
  tearing and interconnecting, Computer Methods in Applied Mechanics and
  Engineering 247-248~(0) (2012) 201 -- 215.

\bibitem{Xu2013}
G.~Xu, B.~Mourrain, R.~Duvigneau, A.~Galligo, Analysis-suitable volume
  parameterization of multi-block computational domain in isogeometric
  applications, Computer-Aided Design 45~(2) (2013) 395 -- 404, solid and
  Physical Modeling 2012.

\bibitem{Juettler2014}
B.~J{\"u}ttler, M.~Kapl, D.-M. Nguyen, Q.~Pan, M.~Pauley, Isogeometric
  segmentation: The case of contractible solids without non-convex edges,
  Computer-Aided Design 57~(0) (2014) 74 -- 90.

\bibitem{Buchegger2015}
F.~Buchegger, B.~J{\"u}ttler, A.~Mantzaflaris, Adaptively refined multi-patch
  {B}-splines with enhanced smoothness, Applied Mathematics and Computation
  (2015) in press.

\bibitem{wang2011converting}
W.~Wang, Y.~Zhang, M.~A. Scott, T.~J. Hughes, Converting an unstructured
  quadrilateral mesh to a standard {T}-spline surface, Computational Mechanics
  48~(4) (2011) 477--498.

\bibitem{wang2012converting}
W.~Wang, Y.~Zhang, G.~Xu, T.~J. Hughes, Converting an unstructured
  quadrilateral/hexahedral mesh to a rational {T}-spline, Computational
  Mechanics 50~(1) (2012) 65--84.

\bibitem{Scott2013}
M.~Scott, R.~Simpson, J.~Evans, S.~Lipton, S.~Bordas, T.~Hughes, T.~Sederberg,
  Isogeometric boundary element analysis using unstructured {T}-splines,
  Computer Methods in Applied Mechanics and Engineering 254~(0) (2013) 197 --
  221.

\bibitem{burkhart2010}
D.~Burkhart, B.~Hamann, G.~Umlauf, Iso-geometric finite element analysis based
  on {C}atmull-{C}lark : Subdivision solids, Computer Graphics Forum 29~(5)
  (2010) 1575--1584.

\bibitem{barendrecht2013}
P.~J. Barendrecht, J.~Shen, J.~Kosinka, M.~Sabin, N.~Dodgson, Isogeometric
  analysis with subdivision surfaces, Eindhoven University of Technology:
  Eindhoven, The Netherlands.

\bibitem{juttler2015}
B.~J{\"u}ttler, A.~Mantzaflaris, R.~Perl, M.~Rumpf, On isogeometric subdivision
  methods for {PDE}s on surfaces, arXiv preprint arXiv:1503.03730.

\bibitem{Doo1978}
D.~Doo, M.~Sabin, Behaviour of recursive division surfaces near extraordinary
  points, Computer-Aided Design 10~(6) (1978) 356 -- 360.

\bibitem{Catmull1978}
E.~Catmull, J.~Clark, Recursively generated {B}-spline surfaces on arbitrary
  topological meshes, Computer-Aided Design 10~(6) (1978) 350 -- 355.

\bibitem{Siqueira2009}
M.~Siqueira, D.~Xu, J.~Gallier, L.~G. Nonato, D.~M. Morera, L.~Velho, Technical
  section: A new construction of smooth surfaces from triangle meshes using
  parametric pseudo-manifolds, Comput. Graph. 33~(3) (2009) 331--340.

\bibitem{Rourke1972}
C.~P. Rourke, B.~J. Sanderson, {Introduction to piecewise-linear topology},
  Springer Study Edition, Springer, Berlin, 1972.

\bibitem{Dokken2013}
T.~Dokken, T.~Lyche, K.~F. Pettersen, Polynomial splines over locally refined
  box-partitions, Computer Aided Geometric Design 30~(3) (2013) 331 -- 356.

\bibitem{Owen1998}
S.~J. Owen, A survey of unstructured mesh generation technology., in: IMR,
  1998, pp. 239--267.

\bibitem{Nguyen2014}
D.-M. Nguyen, M.~Pauley, B.~J{\"u}ttler, Isogeometric segmentation. {P}art
  {II}: On the segmentability of contractible solids with non-convex edges,
  Graphical Models 76~(5) (2014) 426 -- 439, geometric Modeling and Processing
  2014.

\bibitem{beirao2012analysis}
L.~Beir{\~a}o~da Veiga, A.~Buffa, D.~Cho, G.~Sangalli, Analysis-suitable
  {T}-splines are dual-compatible, Computer methods in applied mechanics and
  engineering 249 (2012) 42--51.

\bibitem{Schumi}
L.~L. Schumaker, Spline functions: basic theory, 3rd Edition, Cambridge
  Mathematical Library, Cambridge University Press, Cambridge, 2007.

\bibitem{LLM01}
B.-G. Lee, T.~Lyche, K.~M{\o}rken, Some examples of quasi-interpolants
  constructed from local spline projectors, in: Mathematical methods for curves
  and surfaces ({O}slo, 2000), Innov. Appl. Math., Vanderbilt Univ. Press,
  Nashville, TN, 2001, pp. 243--252.

\bibitem{beirao2014actanumerica}
L.~B. da~Veiga, A.~Buffa, G.~Sangalli, R.~V{\'a}zquez, Mathematical analysis of
  variational isogeometric methods, Acta Numerica 23 (2014) 157--287.

\bibitem{bazilevs2006isogeometric}
Y.~Bazilevs, L.~Beir{\~a}o~da Veiga, J.~Cottrell, T.~Hughes, G.~Sangalli,
  Isogeometric analysis: approximation, stability and error estimates for
  h-refined meshes, Mathematical Models and Methods in Applied Sciences 16~(07)
  (2006) 1031--1090.

\bibitem{da2012anisotropic}
L.~B. Da~Veiga, D.~Cho, G.~Sangalli, Anisotropic {NURBS} approximation in
  isogeometric analysis, Computer Methods in Applied Mechanics and Engineering
  209 (2012) 1--11.

\bibitem{Hebey1996}
E.~Hebey, Sobolev spaces on Riemannian manifolds, Vol. 1635, Springer Science
  \& Business Media, 1996.

\bibitem{Hebey2000}
E.~Hebey, Nonlinear analysis on manifolds: Sobolev spaces and inequalities,
  Vol.~5, American Mathematical Soc., 2000.

\bibitem{dede2015isogeometric}
L.~Ded{\`e}, A.~Quarteroni, Isogeometric analysis for second order partial
  differential equations on surfaces, Computer Methods in Applied Mechanics and
  Engineering 284 (2015) 807--834.

\end{thebibliography}
\end{document}